\documentclass[reqno]{amsart}
\usepackage{amsmath,amsthm,amssymb}
\usepackage{latexsym}
\usepackage{eucal}

\newtheorem{theorem}{Theorem}[section]
\newtheorem{lemma}{Lemma}[section]

\newtheorem{Open question}{Open question}
\newtheorem{proposition}{Proposition}[section]

\theoremstyle{definition}
\newtheorem{definition}{Definition}[section]

\def\cal{\mathcal}
\let\Re=\undefined
\DeclareMathOperator{\Re}{Re}
\let\Im=\undefined
\DeclareMathOperator{\Im}{Im}

\begin{document}
\title[Schr\"odinger operators and associated hyperbolic pencils \ldots
\ldots]{Schr\"odinger operators and associated hyperbolic pencils}
\author{Sergey A. Denisov}
\address{University of Wisconsin-Madison,
Mathematics Department, 480 Lincoln Dr., \linebreak Madison, WI,
53706-1388, USA,  {\rm e-mail: denissov@math.wisc.edu}} \maketitle

\begin{abstract}
For a large class of Schr\"odinger operators, we introduce the
hyperbolic quadratic pencils by making the coupling constant
dependent on the energy in the very special way. For these
pencils, many problems of scattering theory are significantly
easier to study. Then, we give some applications to the original
Schr\"odinger operators including one-dimensional Schr\"odinger
operators with $L^2$-- operator-valued potentials,
multidimensional Schr\"odinger operators with slowly decaying
 potentials.
\end{abstract} \vspace{1cm}

\section{Introduction}

In this paper, we consider two classes of Schr\"odinger operators:
one-dimensional operator with operator-valued potential

\begin{equation}
L=-\frac{d^2}{dr^2}+V(r), r>0
\end{equation}
and the standard
\begin{equation}
H=-\Delta+V(x), x\in \mathbb{R}^3 \label{hamiltonian}
\end{equation}

Operator $L$ can be thought of as $L=L_0+V$, where
\begin{equation}
 L_0= \left[
\begin{array}{cccc}
\displaystyle
-\frac{d^2}{dr^2}  & 0 & 0 & \cdots \\
 0 & \displaystyle -\frac{d^2}{dr^2} & 0 & \cdots \\
 0 & 0 & \displaystyle -\frac{d^2}{dr^2} & \cdots\\
\cdots & \cdots & \cdots & \cdots
 \\
\end{array}\right]
\end{equation}
acts on $\bigoplus\limits_{n=1}^\infty L^2(\mathbb{R}^+)$ with
domain  $\cal{D}(L_0)=\bigoplus\limits_{n=1}^\infty
\dot{H}^2(\mathbb{R}^+)$ (i.e. we consider the Dirichlet boundary
conditions at zero). The selfadjoint $V(r)$ is given by
\begin{equation}
 V(r)= \left[
\begin{array}{cccc}
v_{11}(r) & v_{12}(r) & \cdots \\
\overline{v}_{12}(r) & v_{22}(r) & \cdots \\
\cdots & \cdots & \cdots
\end{array}\right]
\end{equation}
and $\|V(r)\|\in L^\infty(\mathbb{R}^+)$ where the norm is taken
as an operator norm in  $\ell^2$. By general spectral theory, $L$
is essentially selfadjoint with the same domain
$\bigoplus\limits_{n=1}^\infty \dot{H}^2(\mathbb{R}^+)$.

For $H$, we assume $V(x)\in L^\infty(\mathbb{R}^3)$ and then,
again, $\cal{D}(H)=H^2(\mathbb{R}^3)$.

 One of the basic questions of the scattering theory is under what decay assumptions
on  potential $V$ there is a nontrivial a.c. spectrum. The
one-dimensional case suggests that some sort of
$L^2(\mathbb{R}^+)$ condition should be sufficient
\cite{dk,Krein}. The one-dimensional case also has simple
matrix-valued generalization \cite{ls}. In the meantime, the
methods available now have not yet yielded the desired results for
situations considered in this paper.

For both $L$ and $H$, we introduce and study the associated
hyperbolic pencils. Then, we apply obtained estimates to the
original Schr\"odinger operators. The structure of the paper is as
follows. In the second section, we study $L$. The third one
contains the discussion of three-dimensional case. The appendix
contains the proof of Combes-Thomas estimate for Schr\"odinger
pencils.

We will use the following notations: $\langle \xi_1,\xi_2\rangle$
denotes the inner product in $\mathbb{C}^n$,
\[
\ln^- x=
\left\{
\begin{array}{cc}
\ln x, & 0<x<1\\
 0, & x\geq 1
\end{array}
\right.
\]
Also, $\langle x\rangle=(|x|^2+1)^{1/2}$ for any vector $x$. For
function $f(x)$, $f_r(x)$ denotes the radial component of the
gradient and $f_\tau(x)$-- the tangential component, $B$ will
denote nonpositive Laplace-Beltrami operator. For operator $O$,
$\sigma(O)$ will mean the spectrum of $O$.

{\bf Acknowledgements.} This research was supported by Alfred P.
Sloan Research Fellowship, and NSF Grant DMS-0500177.

\section{One-dimensional Schrodinger operators with
operator-valued potential}

Consider the family of operators $L(t)$, given by the coupling
constant $t\in \mathbb{R}$:

\begin{equation}
L(t)=-\frac{d^2}{dr^2}+tV(r), r>0
\end{equation}

\begin{definition}
We say that $\mathbb{R}^+\subseteq \sigma_{ac}(L(t))$ {\it
generically} if this property holds for all $t\in \Omega \subseteq
\mathbb{R}$ where $\Omega$  is a full measure set in $\mathbb{R}$.
\end{definition}

The main goal of this section is to prove the following
\begin{theorem}
Assume that self-adjoint $V(r)$ satisfies $\|V(r)\|\in
L^2(\mathbb{R}^+)\cap L^\infty(\mathbb{R}^+)$. Then,
$\mathbb{R}^+\in \sigma_{ac}(L(t))$ generically. \label{theorem1}
\end{theorem}

Notice carefully, that under the conditions of the Theorem, the
essential spectrum of operator $L$ is not necessarily
$\mathbb{R}^+$. For example, taking off-diagonal elements in
$V(r)$ all equal to zero, one can arrange $v_{kk}(r),
k=1,2,\ldots$ to be such that the spectrum of $L$ is purely a.c.
on $\mathbb{R}^+$, and is, say, dense pure point on some negative
interval.

Consider $F(r)=(f(r),0,\ldots)\in \bigoplus\limits_{n=1}^\infty
L^2(\mathbb{R}^+)$ with $f(r)$-- compactly supported function from
$L^2(\mathbb{R}^+)$, $\|f\|=1$.  Assume that the support of $f$ is
inside the interval $[0,\delta]$. Then, for each $t$, consider the
spectral measure $d\sigma (\lambda, t)$ generated by $F$ and an
operator $L(t)$. Take $\lambda\in [c,d]\subset \mathbb{R}^+$ and
$|t|<T$. Under conditions of the Theorem, we will show that for
generic $t\in [-T,T]$ the following is true: $d\sigma(\lambda,
t)/d\lambda>0$ for a.e. $\lambda\in [c,d]$. That would imply
$[c,d]\subset \sigma_{ac}(H(t))$ for generic $t\in [-T,T]$. Since
$c,d,T$ are arbitrary, the statement of the Theorem follows. But
first we have to obtain some preliminary results.

Consider the family of measures $d\sigma(\lambda,t)$ restricted to
$[c,d]\subset \mathbb{R}^+$.
\begin{lemma}
The measure $d\sigma(\lambda,t)$ is weakly continuous in $t\in
[-T,T]$.
\end{lemma}
\begin{proof}
Indeed, we obviously have $((L(t)-z)^{-1}F,F)\to
((L(t_0)-z)^{-1}F,F)$ for any $z\in \mathbb{C}^+$ as long as $t\to
t_0$. Therefore, by the Spectral Theorem and Weierstrass
approximation argument,
\[
\int h(\lambda)d\sigma(\lambda,t)\to \int
h(\lambda)d\sigma(\lambda,t_0)
\]
for any compactly supported continuous $h(\lambda)$ and $t\to
t_0$.
\end{proof}

The weak continuity allows us to use Riesz Representation Theorem
to correctly define Radon measure $d\nu$ on the set
$\Upsilon=(c,d)\times (-T,T)$ by letting
\[
\int g(\lambda,t)d\nu(\lambda,t)=\int\limits_{-T}^T dt
\int\limits_c^d g(\lambda,t)d\sigma(\lambda,t)
\]
for any continuous $g(\lambda,t)$ supported inside $\Upsilon$. For
each $t$, we have the decomposition
$d\sigma(\lambda,t)=\sigma'(\lambda,t)d\lambda+d\sigma_s(\lambda,t)$.
 On
the other hand, measure $d\nu$ allows decomposition with respect
to two-dimensional Lebesgue measure $d\mu$ on $\Upsilon$:
\[
d\nu(\lambda,t)=\nu'(\lambda,t)d\mu+d\nu_s(\lambda,t)
\]
\begin{lemma}
We have
\[
\nu'(\lambda,t)=\sigma'(\lambda,t)
\]
for $\mu$--a.e. $\lambda,t\in \Upsilon$. Moreover
\[
d\nu_s(\lambda,t)=dtd\sigma_s(\lambda,t)
\]
\label{measure}
\end{lemma}
{\bf Remark.} The last equality is understood in the following
sense
\[
\int g(\lambda,t)d\nu_s(\lambda,t)=\int\limits_{-T}^T dt
\int\limits_c^d g(\lambda,t)d\sigma_s(\lambda,t)
\]
i.e. as equality of Radon measures generated by positive linear
functionals on $C_c(\Upsilon)$.

\begin{proof}
Let us first show that $\sigma'(\lambda,t)$ is measurable with
respect to $d\mu$. To do that, define the Herglotz function
\[
M(z,t)=\int \frac{d\sigma(\lambda,t)}{\lambda-z}, z\in
\mathbb{C}^+
\]
By Spectral Theorem,
\[
M(z,t)=((H(t)-z)^{-1}F,F)
\]
and this function is analytic in $z\in \mathbb{C}^+$ and
continuous in $t\in [-T,T]$. Introduce the set $\Omega$ of
$\lambda\in (c,d), t\in [-T,T]$ for which $\lim_{n\to\infty} \Im
M(\lambda+in^{-1},t)$ exists and is finite. By Cauchy criteria,
\[
\Omega=\bigcap\limits_{j=1}^\infty \bigcup\limits_{N=1}^\infty
\bigcap\limits_{m,k>N} \left\{(\lambda,t): |\Im
M(\lambda+im^{-1},t)-\Im M(\lambda+ik^{-1},t)|<j^{-1}\right\}
\]
and $\Omega$ is Borel. The boundary behavior of Herglotz functions
implies that intersection of $\Omega$ with any line $t=t_0$ has a
full one-dimensional Lebesgue measure. Therefore, by Fubini,
$\Omega$ has the full two-dimensional Lebesgue measure on
$\Upsilon$. Also, for any $(\lambda,t)\in \Omega$, we have
$\pi^{-1}\Im M(\lambda+in^{-1},t)\to \sigma'(\lambda,t)$ (that can
be regarded as the definition of $\sigma'$, which is then equal to
the corresponding maximal function $d\sigma/d\lambda$ Lebesgue
a.e.). Therefore, $\sigma'(\lambda,t)$ is $d\mu$ measurable.
Moreover, since
\[
\int \sigma'(\lambda,t)d\lambda\leq \int\limits
d\sigma(\lambda,t)=1
\]
for any $t$, we have $\sigma'(\lambda,t)\in L^1(\Upsilon)$ by
Fubini. Thus, we are left to show that $dtd\sigma_s(\lambda,t)$ is
$d\mu$--singular. Besides $\Omega$, consider
$\Omega_1=\{(\lambda,t)\in \Upsilon : \Im M(\lambda+in^{-1},t)\to
+\infty \}$ and $\Omega_2=\{(\lambda,t)\in \Upsilon $ for which
$\Im M(\lambda+in^{-1},t)$ has no limit, finite or infinite\}. In
the same way, one can show that $\Omega_{1(2)}$ are Borel. For any
$g(\lambda,t)\in C_c(\Upsilon)$, let
\[
g_n(\lambda,t)=\frac{g(\lambda,t)}{1+\pi^{-1}\Im
M(\lambda+in^{-1},t)}
\]
Clearly, $g_n(\lambda,t)\in C_c(\Upsilon)$ and by definition
\begin{equation}
\int g_n(\lambda,t)d\nu(\lambda,t)=\int\limits_{-T}^T  dt
\int\limits_c^d g_n(\lambda,t)d\sigma(\lambda,t) \label{lem2}
\end{equation}
Consider the l.h.s. The functions $g_n(\lambda,t)$ are uniformly
bounded and converge to $g(\lambda,t)(\sigma'(\lambda,t)+1)^{-1}$
on $\Omega$ and to $0$ on $\Omega_1$. By dominated convergence
theorem,
\[
\int_\Omega g_n d\nu\to \int_\Omega g(\sigma'+1)^{-1}d\nu,
\int_{\Omega_1} g_n d\nu\to 0
\]
For the r.h.s. of (\ref{lem2}), apply dominated convergence
theorem for the inner integral first. When doing that, we take
into account that intersection of $\Omega_2$ with any line $t=t_0$
has zero measure with respect to $d\sigma(\lambda,t_0)$. Also, an
intersection of $\Omega$ with any line $t=t_0$ has zero measure
with respect to $d\sigma_s(\lambda,t_0)$. Therefore, the r.h.s.
converges to
\[
\int_\Upsilon \frac{g\sigma'}{\sigma'+1}d\mu
\]
Comparing the limits, we have
\[
\int_\Upsilon \frac{g\sigma'}{\sigma'+1}d\mu\geq \int_\Omega
g(\sigma'+1)^{-1}d\nu=\int_\Omega g(\sigma'+1)^{-1}\sigma'd\mu+
\]
\[
\int_\Omega g(\sigma'+1)^{-1}dtd\sigma_s(\lambda,t)
\]
Consequently,
\[
\int_\Omega g(\sigma'+1)^{-1}dtd\sigma_s(\lambda,t)=0
\]
for any $g\in C_c(\Upsilon)$. Therefore, $dtd\sigma_s (\Omega)=0$
and $dtd\sigma_s$ is $d\mu$-- singular since $\Omega$ is of the
full Lebesgue measure. The statement of the Lemma now follows from
the uniqueness of the $d\mu$--decomposition for the measure
$d\nu$.
\end{proof}
The main idea of the proof of Theorem \ref{theorem1} is based on
getting the entropy bound for the density of $\nu$, i.e. we will
prove that
\begin{equation}
\int\limits_\Upsilon \ln \sigma'(\lambda,t)d\mu>-\infty
\label{entropy1}
\end{equation}
Since $\sigma'(\lambda,t)\in L^1(\Upsilon)$, an application of
Fubini gives
\[
\int\limits_c^d \ln^- \sigma'(\lambda,t)d\lambda>-\infty
\]
for Lebesgue a.e. $t\in [-T,T]$. Clearly, summability of the
logarithm ensures that the a.c. component of the measure is
supported on $[c,d]$. That implies the statement of the Theorem.

 Thus, we have to show (\ref{entropy1}).

 The proof will be based on the approximation of
operator-valued potential by the matrix-valued ones. Consider
$V_{n,R}(r)=\Pi_n V(r)\cdot \chi_{[0,R]}(r) \Pi_n$, where
$\chi_{\Delta}(r)$ is the characteristic function of the interval
$\Delta$, and $\Pi_n$ is the projection on first $n$ coordinates
in $\ell^2$. Thus, the non-zero part of $V_{n,R}(r)$ is obtained
by cutting $n\times n$ matrix from the upper-left corner of the
matrix representation for $V(r)$ and restricting this
matrix-function to the interval $[0,R]$. Notice that $V_{n,R}(r)$
is self-adjoint,
$\|V_{n,R}(r)\|_{L^\infty(\mathbb{R}^+)}+\|V_{n,R}(r)\|_{
L^2(\mathbb{R}^+)}<C$ uniformly in $R$ and $n$. Moreover, the new
operator $L_{n,R}$ is decoupled into the orthogonal sum of two
operators: the first one, call it $L_{n,R}^1$, is one-dimensional
Schr\"odinger operator with Dirichlet boundary conditions and
$n\times n$ matrix-valued potential $V_{n,R}(r)$. The other
operator is free one-dimensional Schr\"odinger operator acting in
$\bigoplus\limits_{k=n+1}^\infty L^2(\mathbb{R}^+)$. Clearly, the
spectral measure of $F$ with respect to $L_{n,R}$ coincides with
the spectral measure of $(\underbrace{f(r),0,\ldots,0}_{n})$ with
respect to $L_{n,R}^1$. Thus, consider
\begin{equation}
L_{n,R}^{1}(t)=-\frac{d^2}{dr^2}I_{n\times n}+tV_{n,R}(r)
\label{lnr}
\end{equation}
with Dirichlet boundary conditions at zero and
\[
V_{n,R}(r)=  \left[
\begin{array}{ccccc}
v_{11}(r) & v_{12}(r) & \cdots & v_{1n}(r)\\
\overline{v}_{12}(r) & v_{22}(r) & \cdots & v_{2n}(r) \\
\cdots & \cdots & \cdots & \cdots \\
\overline{v}_{1n}(r) & \overline{v}_{2n}(r) & \cdots & v_{nn}(r)
\end{array}\right]\cdot \chi_{[0,R]}(r)
\]
Let $d\sigma_{n,R}(\lambda,t)$ be the spectral measure of $F$ with
respect to $L_{n,R}^{1}(t)$. Since $V_{n,R}(r)$ is compactly
supported, $d\sigma(\lambda,t)=\sigma'(\lambda,t)d\lambda dt$ with
$\sigma'(\lambda,t)$-- smooth in $(\lambda,t)\in \Upsilon$. We
will prove
\begin{equation}
\int\limits_\Upsilon \ln \sigma_{n,R}'(\lambda,t)d\mu>C
\label{entropy2}
\end{equation}
uniformly in $n, R$. Then, the standard argument with weak
convergence of $d\sigma_{n,R}(\lambda,t)$ to $d\sigma(\lambda,t)$
and weak upper semicontinuity of the entropy (\cite{ks}, Corollary
5.3) will imply (\ref{entropy1}).

We will need several simple and well-known facts (Lemmas
\ref{lemmanew}--\ref{lemma3}). Consider
\[
L(t)=-\frac{d^2}{dr^2}+tQ(r)
\]
with Dirichlet boundary conditions and $n\times n$ matrix-function
$Q(r)=Q^*(r)\in L^2(\mathbb{R}^+)$ compactly supported on $[0,R]$.
Consider $u(r,k,t)=(L(t)-k^2-i(+0))^{-1}F$, the restriction of the
solution to the real axis ($k\neq 0$). The potential is compactly
supported so this restriction clearly exists and has the following
asymptotics
\[
u(r,k,t)=\exp(irk)A(k,t), r>R
\]
where $A(k,t)$ is a vector. Moreover, there is the unique $u$ that
both solves equation,  satisfies boundary condition and
asymptotics at infinity. Let $\sigma(\lambda,t)$ be the spectral
measure of $F$ with respect to $L(t)$.
\begin{lemma}
Let $d\sigma(\lambda,t)$ be the spectral measure of $F$ with
respect to $L(t)$. Then,
\begin{equation}
\sigma'(\lambda,t)=k\pi^{-1}\|A(k,t)\|^2, \lambda=k^2, k>0
\label{factor1}
\end{equation}
\label{lemmanew}
\end{lemma}
\begin{proof}
Since the potential is finitely supported, we have
\[
\sigma'(\lambda,t)=\pi^{-1} \Im \int\limits_0^\infty \langle
u(r,k,t), F(r)\rangle dr
\]
by the Spectral Theorem. On the other hand, from equation
$-u''+tQu=k^2u+F$ we have
\[
-\langle u,u''\rangle +t\langle u, Q u\rangle=k^2\langle u,
u\rangle +\langle u,F\rangle
\]
Taking imaginary part, integrating over $\mathbb{R}^+$, and using
the boundary condition and asymptotics, we get
\[
\Im \int\limits_0^\infty \langle u(r,k,t), F(r)\rangle
dr=k\|A(k,t)\|^2
\]

\end{proof}

We introduce now the standard object in the scattering theory, the
Jost solution. Let $k\in \mathbb{R}, k\neq 0$ and $J(r,k,t)$ be
the solution to
\[
-J''+tQJ=k^2J,\quad J(r,k,t)=\exp(ikr), r>R
\]
One can easily show existence and uniqueness of $J$. Also, let
$\alpha(r,k,t)$ be the solution to Cauchy problem
\begin{equation}
-\alpha''+tQ\alpha=k^2\alpha, \quad \alpha(0,k,t)=0,\quad
\alpha'(0,k,t)=1 \label{alpha0}
\end{equation}
\begin{lemma}
We have
\begin{equation}
 u(r,k,t)=-J(r,k,t)\int\limits_0^r
G_{12}(\rho,k,t) F(\rho)d\rho+\alpha(r,k,t)\int\limits_r^\infty
G_{22}(\rho,k,t) F(\rho)d\rho \label{green}
\end{equation}
where
\[
\left[
\begin{array}{cc}
G_{11}  & G_{12}\\
G_{21} & G_{22}
\end{array}
\right]=\left[
\begin{array}{cc}
J  & \alpha\\
J' & \alpha'
\end{array}
\right]^{-1}
\]
\end{lemma}
\begin{proof}
The proof is a simple calculation.
\end{proof}
The formula for the inverse is given by
\begin{lemma}
The following identity is true
\[
\left[
\begin{array}{cc}
G_{11}  & G_{12}\\
G_{21} & G_{22}
\end{array}
\right]=
 \left[
\begin{array}{cc}
J^{-1}(0,k,t)  & 0\\
-2ik|J^{-1}(0,k,t)|^{2} & -[J^*(0,k,t)]^{-1}
\end{array}
\right]
 \left[
\begin{array}{cc}
\left[\alpha^*(r,k,t)\right]' & -\alpha^*(r,k,t)\\
\left[J^*(r,k,t)\right]' & -J^*(r,k,t)
\end{array}
\right]
\]
\end{lemma}
\begin{proof}
The proof is a simple calculation that uses identity
$(Y_1^*)'Y_2-Y_1^* (Y_2)'=const$ which is true for any $Y_{1(2)}$
that solve $-Y''+tQY=k^2Y$.
\end{proof}
In the previous Lemma, the invertibility of $J(0,k,t)$ follows,
for instance, from well-known formulas
\begin{lemma}
If
\[
\mathfrak{A}(k,t)=(J(0,k,t)+(ik)^{-1}J'(0,k,t))/2,
\mathfrak{B}(k,t)=(J(0,k,t)-(ik)^{-1}J'(0,k,t))/2
\]
then
\[
J(0,k,t)=\mathfrak{A}(k,t)+\mathfrak{B}(k,t),
|\mathfrak{A}(k,t)|^2=I+|\mathfrak{B}(k,t)|^2
\]
\label{lemma3}
\end{lemma}
\begin{proof}
The second formula follows from the identity
$J^*(J)'-(J^*)'J=2ik$.
\end{proof}

Using previous Lemmas, we have
\begin{equation}
A(k,t)=J^{-1}(0,k,t) \hat{F}(k,t)
\end{equation}
where
\begin{equation}
\hat{F}(k,t)=\int\limits_0^\delta \alpha^*(\rho,k,t) F(\rho)d\rho
\end{equation}
Recalling (\ref{factor1}),
\begin{equation}
\sigma'(\lambda,t)=k\pi^{-1} \|J^{-1}(0,k,t) \hat{F}(k,t)\|^2,
\lambda=k^2
\end{equation}
The function $\hat{F}(k,t)$ has analytic continuation in $k$ to
$\mathbb{C}$ and depends on $Q$ on $[0,\delta]$. Therefore, only
$J^{-1}(0,k,t)$ is responsible for scattering properties.

 To study $J(0,k,t)$, we will use the following argument. Instead of dealing with
the standard Schr\"odinger equation
\[
-J''(r,k,t)+tQ(r)J(r,k,t)=k^2J(r,k,t)
\]
we will consider
\begin{equation}
-D''(r,k,\xi)+k\xi Q(r)D(r,k,\xi)=k^2 D(r,k,\xi),\xi\in \mathbb{R}
\label{pencil1}
\end{equation}
Thus, we make the coupling constant energy-dependent. We can
single out $D(r,k,\xi)$ as the solution satisfying the same Jost
asymptotics at infinity
\[
 D(r,k,\xi)=\exp(ikr), r>R
\]
 Obviously,  we have
\begin{equation}
J(0,k,kt)=D(0,k,t)
\end{equation}
and, consequently,
\begin{equation}
\sigma'(k^2,kt)=k\pi^{-1}\|D^{-1}(0,k,t)\hat{F}(k,kt)\|^2
\label{factor2}
\end{equation}

Now, the main advantage of dealing with $D(0,k,t)$ instead of
$J(0,k,t)$ is that it allows analytic continuation in $k$ to the
upper half-plane along with the nice uniform estimates. Indeed,
consider equation
\begin{equation}
-D''+k\xi QD=k^2D \label{pencil2}
\end{equation}
for $k\in \mathbb{C}^+$ and look for $D(r,k,\xi)$ which satisfies
Jost condition at infinity, i.e. $D(r,k,\xi)=\exp(ikr), r>R$. This
$D$ can be easily obtained in the following fashion. Write
(\ref{pencil2}) as a system

\begin{equation}
Y'= \left[
\begin{array}{cc}
0  & 1\\
k\xi Q -k^2 & 0
\end{array}
\right]Y \label{system}
\end{equation}
 Introduce
\[
Y_0=\left[
\begin{array}{cc}
\exp(ikr) & \exp(-ikr)\\
ik\exp(ikr)  & -ik\exp(-ikr)
\end{array}
\right],
U_{1(2)}(r_1,r_2,\xi)=\overset{\curvearrowleft}{\int_{r_1}^{r_2}}
\exp \left[\mp \frac{i\xi }{2} Q(s)ds\right]
\]
\[
 U(r,\xi)=\left[
\begin{array}{cc}
U_1(0,r,\xi) & 0 \\
0  & U_2(0,r,\xi)
\end{array}\right]
\]
Then, for $S=U^{-1}Y_0^{-1}Y$, we have
\[
S'=\left[
\begin{array}{cc}
0 & - A^*(r)\exp(-2ikr) \\
- A(r) \exp(2ikr)  & 0
\end{array}
\right]S
\]
where
\begin{equation}
A(r,\xi)=-\frac{i\xi}{2} U_2^{-1}(0,r,\xi)Q(r)U_1(0,r,\xi)
\label{coeff}
\end{equation}
By letting $S(\infty)=(I,0)^t$, we have
 \[
 S_1(r,k,\xi)=1+\int\limits_r^\infty A^*(s_1,\xi)
 \int\limits_{s_1}^\infty \exp \left[ 2ik(s_2-s_1)\right]
 A(s_2,\xi)S_1(s_2,k,\xi)ds_2ds_1
 \]
Gronwall-Bellman's Lemma yields
\[
\|S_1(r,k,\xi)\|\leq \exp\left[ \frac{\xi^2}{8\Im
k}\int\limits_r^\infty \|Q(s)\|^2ds \right]
\]

and for
\[
S_2(r,k,\xi)=\int\limits_r^\infty A(s,\xi)\exp(2iks)S_1(s,k,\xi)ds
\]
we have
\[
\|S_2(r,k,\xi)\|\leq \frac{|\xi|}{2}\exp\left[ \frac{\xi^2}{8\Im
k}\int\limits_0^\infty \|Q(s)\|^2ds \right] \int\limits_r^\infty
\|Q(s)\|\exp[-2(\Im k) s]ds
\]

Clearly, we can express $D(r,k,\xi)$ in the following way
\[
D(r,k,\xi)=\left[\exp(ikr)U_1(0,r,\xi)S_1(r,k,\xi)+\exp(-ikr)U_2(0,r,\xi)S_2(r,k,\xi)\right]
U_1^{-1}(0,R,\xi)
\]
Therefore,  obviously, we have existence, analyticity, and
continuity of $D(r,k,\xi)$ for any $k\in \overline{\mathbb{C}^+}$
(remember that $Q$ is compactly supported). Moreover, the
following uniform estimate holds
\begin{lemma}
\begin{equation}
\|D(0,k,\xi)\|\leq \exp\left[ \frac{\xi^2}{8\Im
k}\int\limits_0^\infty \|Q(s)\|^2ds \right] \cdot
\left[1+\frac{|\xi|}{2\sqrt{2\Im k}}\|Q\|_2\right]
\label{estimatel2}
\end{equation}
holds true for any $k\in \mathbb{C}^+$ and any $Q\in
L^2(\mathbb{R}^+)$. Also,
\[
\|D^{-1}(0,k,\xi)\|<C,\quad \Im k>\kappa
\]
where $C$  and $\kappa$ depend on $\xi$ and $\|Q\|_2$ only.
\label{estimate1}
\end{lemma}
We have analogous estimate from above on $\|D'(0,k,\xi)\|$.

The next Lemma yields the quantitative version of invertibility of
$D(0,k,\xi)$. Introduce $\mu(r,k,\xi)=D(r,k,\xi)\exp(-ikr)$. Then,
we have
\begin{lemma}
The following identity is satisfied for any $k\in \mathbb{C}^+$
\[
|D^{-1}(0,k,\xi)|^{2}+[D^*(0,k,\xi)]^{-1}\left[\frac{\Im k}{|k|^2}
\int\limits_0^\infty |\mu'(s,k,\xi)|^2ds\right] \cdot
D^{-1}(0,k,\xi)=
\]
\[
=\Im \left[ \frac{D'(0,k,\xi)D^{-1}(0,k,\xi)}{k}\right]
\]
\end{lemma}

\begin{proof}
For $\mu$:
\[
\mu''(r,k,\xi)+2ik\mu'(r,k,\xi)=k\xi Q(r)\mu(r,k,\xi),\quad
\mu(r,k,\xi)=1, r>R
\]
Divide the both sides by $2ik$, multiply from the left by
$\mu^*(r,k,\xi)$ and integrate from $0$ to $\infty$. Taking the
real part, we have
\[
I-|\mu(0,k,\xi)|^2+\frac{\Im k}{|k|^2} \int\limits_0^\infty
|\mu'(s,k,\xi)|^2ds=\frac{1}{2ik}\mu^*(0,k,\xi)\mu'(0,k,\xi)
-\frac{1}{2i\bar{k}}{\mu^*}'(0,k,\xi)\mu(0,k,\xi)
\]
Clearly, the last identity shows that $\mu(0,k,\xi)$ is invertible
for any $k\in \overline{\mathbb{C}^+}\backslash \{0\}$. Moreover,
\[
|\mu^{-1}(0,k,\xi)|^{2}+[\mu^*(0,k,\xi)]^{-1}\left[\frac{\Im
k}{|k|^2} \int\limits_0^\infty |\mu'(s,k,\xi)|^2ds\right] \cdot
\mu^{-1}(0,k,\xi)=
\]
\[
=I+\Re \left[\frac{\mu'(0,k,\xi) \mu^{-1}(0,k,\xi)}{ik}\right]=\Im
\left[ \frac{D'(0,k,\xi)D^{-1}(0,k,\xi)}{k}\right]
\]
since $\mu(r,k,\xi)=D(r,k,\xi)\exp(-ikr)$.
\end{proof}
As a simple corollary, we get that the matrix-valued function

\begin{equation}
G(k)=\frac{D'(0,k,\xi)D^{-1}(0,k,\xi)}{k} \label{fg}
\end{equation}
is Herglotz and its boundary value on the real line is factorized
through $|D^{-1}(0,k,\xi)|^2$. Moreover, we have a uniform bound
on $G(k)$ for large $\Im k$ due to Lemma \ref{estimate1}, which
yields $\|G(k)\|<C$ for large $\Im k$ (where $C$ depends only on
$\xi$ and $\|Q\|_2$). As a corollary from the integral
representation for Herglotz function, we have
\begin{lemma}
For any $k\in \mathbb{C}^+$ and $\|Q\|\in L^2(\mathbb{R}^+)\cap
L^\infty (\mathbb{R}^+)$, we have
\begin{equation}
\|D^{-1}(0,k,\xi)\|\leq C(\xi, \|Q\|)\left[\Im k\right]^{-1/2}
(|\Re k|+1) \label{trivial}
\end{equation}
\end{lemma}

One may wonder why the function $D(r,k,\xi)$ possesses so many
properties and may be there is some algebraic fact behind it. The
partial answer to that question is contained in the following
Lemma

\begin{lemma}
Let matrix $Y(r,k,\xi)$ solve (\ref{system}) and
\[
E(r,k)= \left[
\begin{array}{cc}
\exp(-2ikr) & 0 \\
0  & 1
\end{array}\right]
\]
If $X$ is defined by
\[
Y=Y_0 U E X
\] then it solves the matrix-valued Krein system
\[
X'(r,\tau,\xi)=\left[
\begin{array}{cc}
i\tau & -A^*(r,\xi) \\
-A(r,\xi)  & 0
\end{array}\right]X(r,\tau,\xi)
\]
with $\tau=2k$ and $A(r,\xi)$ defined by (\ref{coeff}).
\label{lemma-alpha}
\end{lemma}
\begin{proof}
The proof is an elementary calculation.
\end{proof}

The matrix-valued Krein systems were studied (see, e.g.,
\cite{Sakh}). For the scalar case, see \cite{Krein, Den1}. One can
express $D(r,k,\xi)$ through the certain special solutions of the
Krein systems that are know to have properties similar to those
established in previous Lemmas.

Consider (\ref{factor2}). As was mentioned before, $\hat{F}(k,kt)$
has analytic continuation in $k$ to $\mathbb{C}$. The following
Lemma is elementary
\begin{lemma}
Fix any $k\in \mathbb{C}^+$ and $T_1>0$. Then, there is
$\delta(k,T_1,\|Q\|_2)>0$ small enough so that there is a function
$F(r)=(f(r), 0, 0, \ldots)$, supported on $[0,\delta]$, for which
\begin{equation}
\|\hat{F}(k,kt)\|>C>0, \forall t\in [-T_1,T_1] \label{estimate2}
\end{equation}
where the constant $C$ depends on $k$, $T_1$, and $\|Q\|_2$.
\label{lemma11}
\end{lemma}
\begin{proof}
 For $\alpha^*(r,k,kt)$ from (\ref{alpha0}), we have
\begin{eqnarray}
\alpha^*(r,k,kt)=\frac{\sin(rk)}{k}+t\int\limits_0^r
{\sin[k(r-\rho)]}
 \alpha^*(\rho,k,kt) Q(\rho)d\rho
\end{eqnarray}
That integral equation can be used to define analytic continuation
in $k$. Assume $k$ is fixed and $\delta\to 0$. Then,
\begin{equation}
\alpha^*(r,k,kt)=r(1+\bar{o}(1)), \quad 0<r<\delta \label{alpha}
\end{equation}
uniformly in $|t|<T_1$. Therefore, to satisfy (\ref{estimate2}),
it is sufficient to choose small $\delta$ and any nonnegative
function $f(r)$ supported on $[0,\delta]$.

\end{proof}

{\it Proof of Theorem \ref{theorem1}.} Fix any $T>0$ and
$[a,b]\subset \mathbb{R}^+$. Let us show that $[a^2,b^2]\subset
\sigma_{ac}(L(t))$ for generic $t\in [-T,T]$. For any $n,R$,
consider $L_{n,R}^{1}$ given by (\ref{lnr}).  Now, the potential
$tV_{n,R}(r)$ in $L_{n,R}^{1}$ is $n\times n$ matrix-function with
compact support. Also, $\|V_{n,R}\|_2\leq \|V\|_2$ for all $n,R$.
Therefore, Lemmas \ref{lemma3}--\ref{lemma11} are applicable.
Consider isosceles triangle in $\mathbb{C}^+$ with base
$I=[a_1,b_1]\supset [a,b]$, sides $I_{1(2)}$, and the adjacent
angles both equal to $\pi/\gamma$. Fix some $k_0\in \mathbb{C}^+$
inside this triangle. Take the function $F(r)$ given by
Lemma~\ref{lemma11} applied to $k_0$ and some large $T_1(T)$ to be
specified later. Let $d\sigma(\lambda,t)$ be the spectral measure
of $F(r)$ corresponding to $L(t)$. We will show that for generic
$t\in [-T,T]$ we have $\sigma'(\lambda,t)>0$ for a.e. $\lambda\in
[a^2,b^2]$.  Let $d\sigma_{n,R}(\lambda)$ be the spectral measure
of $F(r)$ with respect to $L_{n,R}^{1}$. By (\ref{factor2}),
\begin{equation}
\sigma'_{n,R}(k^2,kt)=k\pi^{-1}\|D^{-1}_{n,R}(0,k,t)\hat{F}_{n,R}(k,kt)\|^2
\label{factor3}
\end{equation}

For each $n,R$ and real $k$, we have factorization
\[
\Im G_{n,R}(k)=|D_{n,R}^{-1}(0,k,\xi)|^2
\]
where $G_{n,R}(k)$ is Herglotz matrix-valued function given by
(\ref{fg}) with uniform in $n,R$ estimates for large $\Im k$.
Consequently,
\[
\int \frac{\Im G_{n,R}(k)}{k^2+1}dk<C
\]
for all $n,R$. Since $\hat{F}_{n,R}(k,kt)$ is entire in $k$ (with
uniform estimates on Taylor coefficients), we have
\begin{equation}
\int_J \sigma'_{n,R}(k^2,kt)dk <C(J,T_1) \label{variation}
\end{equation}
uniformly in $n,R, |t|<T_1$ for any interval $J\subset
\mathbb{R}$.

 Consider function
\[
g_{n,R}(k)=\ln \|D^{-1}_{n,R}(0,k,t)\hat{F}_{n,R}(k,kt)\|
\]
Since $D^{-1}_{n,R}(0,k,t)\hat{F}_{n,R}(k,kt)$ is analytic in
$\mathbb{C}^+$ and continuous down to the real line, $g_{n,R}(k)$
is subharmonic. The mean value inequality applied to $g_{n,R}(k)$
at point $k_0$ yields
\[
\int_{I} g_{n,R}(k)\omega(k,k_0)dk+\int_{I_{1(2)}}
g_{n,R}(k)\omega(k,k_0)dk\geq g_{n,R}(k_0)
\]
where $\omega(k,k_0)$ is the Green function for our
triangle\footnote{Analogous trick was used in \cite{killip}.}. It
is well-known that $\omega(k,k_0)$ is smooth, positive inside $I,
I_{1(2)}$, and vanishes at the vertices of triangle  such that
$\omega(k,k_0)\leq C|k-a_1(b_1)|^{\gamma-1}$. At $k_0$, we have
\[
g_{n,R}(k_0)\geq \ln \left[
\|D_{n,R}(0,k_0,t)\|^{-1}\|\hat{F}_{n,R}(k_0,t)\|\right]>C
\]
uniformly in $n$ and $R$ due to (\ref{estimatel2}) and
(\ref{estimate2}). By (\ref{trivial}) and trivial estimate on
$\hat{F}(k)$ from above, we have
\[
\int_{I_{1(2)}} g_{n,R}(k)\omega(k,k_0)dk<C
\]
uniformly in $n$ and $R$. In the last inequality, we also used the
properties of the weight $\omega$.

Consequently, we have
\begin{equation}
\int_{I} g_{n,R}(k)\omega(k,k_0)dk>C>-\infty
\end{equation}
uniformly in $n,R$. By (\ref{factor3}) and (\ref{variation}),
\[
\int\limits_a^b \ln^- \sigma'_{n,R}(k^2,kt)dk>C,\quad
\int\limits_{a^2}^{b^2} \ln^- \sigma_{n,R}'(\lambda,
t\sqrt\lambda)d\lambda>C
\]
where the last inequality is satisfied uniformly in $n, R, t\in
[T_1,T_1]$. Integration in $t$ yields
\begin{equation}
\int\limits_{a^2}^{b^2}\int\limits_{-T_1 a}^{T_1 a} \ln^-
\sigma'_{n,R}(\lambda,t)d\lambda dt>C \label{uniform}
\end{equation}
uniformly in $n,R$ by simple change of variables. Take
$T_1=a^{-1}T$.

Now, we consider the two-dimensional measures $d\sigma(\lambda,t)$
and $d\sigma_{n,R}(\lambda,t)$, both restricted to
$[a^2,b^2]\times [-T,T]$. It is easy to show that
$d\sigma_{n,R}(\lambda,t)\to d\sigma(\lambda,t)$ in the weak-star
sense. Therefore, the weak upper semicontinuity of the entropy
(see \cite{ks}, p. 293) and estimate (\ref{uniform}) imply
\[
\int\limits_{a^2}^{b^2}\int\limits_{-T}^{T} \ln^- \left(
\frac{d\sigma}{d\mu}\right) d\lambda dt>-\infty
\]
On the other hand, Lemma \ref{measure} implies that
$d\sigma/d\mu=\sigma'(\lambda,t)$. Consequently,
\[
\int\limits_{a^2}^{b^2}\int\limits_{-T}^{T} \ln^-
\sigma'(\lambda,t) d\lambda dt>-\infty
\]
Therefore, by Fubini theorem,
\[
\int\limits_{a^2}^{b^2} \ln^- \sigma'(\lambda,t) d\lambda
dt>-\infty
\]
for a.e. $t\in [-T,T]$. That, of course, implies $[a^2,b^2]\subset
\sigma_{ac}(L(t))$ for generic $t\in [-T,T]$.
 $\Box$

{\bf Remark.} We proved that the function
\[
p(t)=\int\limits_{I} \ln \sigma'(\lambda,t) d\lambda
\]
belongs to $L^1_{\rm loc}(\mathbb{R})$ for any
$I\subset\mathbb{R}^+$. Another simple property of $p(t)$ is upper
semicontinuity. It follows from the weak continuity of
$d\sigma(\lambda,t)$ with respect to $t$ and weak upper
semicontinuity of the entropy. Therefore, the set of ``good" $t$
for which $p(t)$ is finite is necessarily $F_\sigma$. We believe
that the statement of the Theorem \ref{theorem1} holds for all
$t$. One can try to prove that by establishing the asymptotics of
the Green functions as $r\to\infty$. Let $Y(r,k)$ be solution to
\[
-Y''+QY=k^2Y
\]
for $k\in \mathbb{C}^+$ that decays at infinity. If
$Y=\exp(ikr)\mu$, we have
\[
\mu''+2ik\mu'=Q\mu
\]
We try to find the solution in the form
\[
Z=\mu'\mu^{-1}
\]
Then
\[
Z'+2ikZ=Q-Z^2
\]
and
\[
Z(r)=Z_0(r)+\int\limits_r^\infty \exp(2ik(s-r)) Z^2(s)ds,\,
Z_0(r)=-\int\limits_r^\infty Q(s)\exp(2ik(s-r))ds
\]
For $\Im k$ large enough, this integral equation can be solved by
contraction argument and that gives us $Z=Z_0+Z_1$, where
$\|Z_0\|\in L^2(\mathbb{R}^+)$ and $\|Z_1\|\in L^1(\mathbb{R}^+)$.
For $\mu$,  we have
\[
\mu'=(Z_0+Z_1)\mu
\]
Unfortunately, the asymptotical analysis of this equation does not
seem to be possible even in matrix-valued case although $Z_0$ is
precise and $\|Z_1\|\in L^1(\mathbb{R}^+)$. That explains why we
have to switch to different problem with energy dependent coupling
constant. For this new problem, the semigroup generated by $Z_0$
happens to be bounded and the usual asymptotical analysis works.

 We could have studied equation
 \[
 -y''+k\xi Qy=k^2y
 \] in the framework of the spectral theory for quadratic hyperbolic
 pencils.

Consider the following quadratic pencil \cite{markus}
\[
P(k)=A_1+kA_2-k^2, k\in \mathbb{C}
\]
where $A_1=-{d^2}/{dr^2}\cdot I_{n\times n}$ with Dirichlet
boundary condition at zero, and $A_2=\xi Q(r)$. Notice that $P(k)$
is hyperbolic (\cite{markus}, p. 169) since the quadratic
polynomial
\[
( P(k)G,G) =\int\limits_0^\infty
\|G'(r)\|^2dr+k\xi\int\limits_0^\infty \langle
Q(r)G(r),G(r)\rangle dr-k^2
\]
has two distinct real roots for any $G(r)\in
\cal{D}(P(k))=\bigoplus\limits_{k=1}^n \dot{H}(\mathbb{R}^+),
\|G\|=1$

The general spectral theory of these pencils ensures invertibility
of $P(k)$ for any $k\in \mathbb{C}, k\notin \mathbb{R}$. That is
another explanation to the fact that function $D(r,k,\xi)$ is
well-defined and invertible for all $k\in \mathbb{C}^+$. Notice
that for Schr\"odinger operators, the Jost function $J(0,k,t)$ can
be degenerate at some points $k_j=i\kappa_j$ that correspond to
negative eigenvalues $-\kappa_j^2$. By Lemma \ref{lemma-alpha},
the study of $P(k)$ is essentially equivalent to analysis of the
corresponding Krein systems and vice versa. Since Krein systems
are understood rather well, we do not pursue any further analysis
of $P(k)$. We just want to mention that matrix-valued Krein
systems are essentially equivalent to matrix-valued Dirac
operators. The $L^2$--conjecture for Dirac operators was resolved
before \cite{den3} and the result obtained was much stronger than
that of Theorem \ref{theorem1}.

Pencils similar to $P(k)$ were studied before, especially for the
purpose of solving the inverse problems (see, e.g. \cite{nabiev},
and references there).

\section{Multidimensional Schr\"odinger operator and corresponding
pencils}

In this section, we consider operator $H$ given by
(\ref{hamiltonian}). For simplicity, we will work in the
three-dimensional case. The $L^2$--conjecture for this case
\cite{simon} reads
\begin{equation}
\int_{\mathbb{R}^3}\frac{V^2(x)}{|x|^2+1}dx<\infty \label{l2}
\end{equation}
and one expects $\mathbb{R}^+\subseteq \sigma_{ac}(H)$. This
problem attracted a lot of attention recently and was resolved
only for some special cases \cite{dd1, dd2, dd3, dd4, dd5, lns1,
lns2, per}. Basically, the main technical difficulty is absence of
thorough asymptotical analysis for the Green function at complex
energies. The operator-valued one-dimensional Schr\"odinger
operator is a toy model for $H$ since it can be written as
\begin{equation}
-\frac{d^2}{dr^2}-\frac{B}{r^2}+V(r,\theta) \label{spher}
\end{equation}
in spherical coordinates with $B$ being  Laplace-Beltrami operator
on the unit sphere $\Sigma$, $\theta\in \Sigma$. For general
operator-valued case, the asymptotics at complex energies is not
obtained (see discussion in the previous section). Of course,
equation (\ref{spher}) is more complicated since $B$ is unbounded.

  In this
paper, we make another step toward understanding of the problem.
Consider $H(t)$ with potential $V$ and  the coupling constant $t$.
Our technique will allow us to easily prove the following results
\begin{theorem}
Assume
\[
V(x)=div\, L(x)
\]
where smooth vector field $L(x)$ satisfies
\[
L(x), |\nabla L(x)|\in L^\infty (\mathbb{R}^3), \quad
\int_{\mathbb{R}^3} \frac{|L(x)|^2}{|x|^2+1}dx <\infty
\]
Then for generic $t$, $\mathbb{R}^+\subseteq
\sigma_{ac}(H(t))$.\label{theorem2}
\end{theorem}

\begin{theorem}
Assume $V(x)$ is bounded and
\[
\int\limits_1^\infty r|v(r)|^2<\infty
\]
for $v(r)=\sup_{|x|=r} |V(x)|$. Then for generic $t$,
$\sigma_{ac}(H(t))=\mathbb{R}^+$.\label{theorem3}
\end{theorem}

 Denote by
$H_{R}(t)$ the Schr\"odinger operator with potential
$tV_R(x)=tV(x)\cdot \chi_{|x|<R}(x)$. For fixed $f(x)\in
L^2(\mathbb{R}^3)$ with compact support inside the unit ball,
introduce the spectral measures $d\sigma(\lambda,t)$ and
$d\sigma_{R}(\lambda,t)$. For three-dimensional case, we have
direct analogs of Lemmas proved in the last section. In particular
(\cite{dd1}, p. 3974)
\begin{lemma}
Assume that $V(x)$ is real-valued compactly supported potential
and $u(x,k,t)=(-\Delta+tV-k^2-(+0)i)^{-1}f$ is the restriction of
the solution to real $k$. Then, for $u(x,k)$, the following
asymptotics holds true
\[
u(x,k,t)=\frac{\exp(ik|x|)}{|x|}\left[A(k,\theta,t)+\bar{o}(1)\right],
\quad \theta=x/|x|
\]
as $|x|\to\infty$. Moreover,
\begin{equation}
\sigma'(k^2,t)=k\pi^{-1} \|A(k,\theta,t)\|^2_{L^2(\Sigma)}
\label{factor33}
\end{equation}
where $d\sigma$ is the spectral measure of $f$ with respect to
$-\Delta+tV$.\label{lemma-factor}
\end{lemma}

Each $A_R(k,\theta,t)$ has analytic continuation to $\mathbb{C}^+$
(besides points corresponding to negative discrete spectrum) but
we are not able to prove any  estimates uniform in $R$ assuming
only  $|V(x)|<C(|x|+1)^{-1+}$. Therefore, instead of dealing with
Schr\"odinger operator, we will consider the corresponding pencil
given by

\[
P(k)=A_1+kA_2-k^2, k\in \mathbb{C}
\]
where $A_1=-\Delta$, $A_2=\xi V(x)$, $\xi\in \mathbb{R}$. Under
the general assumption $V(x)\in L^\infty(\mathbb{R}^3)$, $P(k)$ is
well-defined for any $k\in \mathbb{C}$ with
$\cal{D}(P(k))=H^2(\mathbb{R}^3)$. One can first define $P(k)$  on
the Schwarz space. Then it is an easy exercise to show that $P(k)$
admits the closure which gives rise  to the operator defined on
$H^2(\mathbb{R}^3)$. Moreover, one can show that
$P^*(k)=P(\bar{k})$ and that pencil $P(k)$ is hyperbolic.

\begin{lemma}
Let  $V(x)\in L^\infty(\mathbb{R}^3)$. Then, for any $k\notin
\mathbb{R}$, $P(k)$ is invertible. If $\psi(x,k,\xi)=P^{-1}(k)f,
k\notin \mathbb{R}$, then
\begin{equation}
\|\psi\|\leq \left|\Im k\right|^{-2} \|f\| \label{spec}
\end{equation}
\label{trifle}
\end{lemma}

\begin{proof}
This is a general fact of spectral theory for hyperbolic quadratic
pencils. Let $k\notin \mathbb{R}$. For any $f\in
H^2(\mathbb{R}^3)$, consider
\[
\langle P(k)f,f \rangle =\int |\nabla f|^2dx +k\xi\int V
|f|^2dx-k^2\int |f|^2dx=-(k-k_1)(k-k_2)\|f\|_2^2
\]
where $k_{1(2)}$ -- real roots. Consequently,
\begin{equation}
\|P(k)f\|\cdot \|f\|\geq |(P(k)f,f)|\geq \left|\Im
k\right|^2\|f\|_2^2 \label{small-i}
\end{equation}
which implies that Ker$P(k)=0$ and $P^{-1}(k)$ is bounded.
Ran$P(k)$ is dense in $L^2(\mathbb{R}^3)$ since
$P^*(k)=P(\bar{k})$ and Ker$P(\bar{k})=0$. Then, (\ref{small-i})
ensures that Ran$P(k)=L^2(\mathbb{R}^3)$ since $P(k)$ is closed.
\end{proof}

Now, assume that $V(x)$ is compactly supported. Then,
$\psi(x,k,\xi)$ can be continued in $k$ down to the real line by
following, e.g., the proof of Agmon's absorption principle
(\cite{rs}, Chapter 13, sect. 8). Then, we have asymptotics
\[
\psi(x,k,\xi)=\frac{\exp(ik|x|)}{|x|}\left[J(k,\theta,\xi)+\bar{o}(1)\right],
\quad \theta=x/|x|
\]
as $|x|\to\infty$ for any $k\in \overline{\mathbb{C}^+}$. From
(\ref{factor33}) and obvious identity
$A(k,\theta,kt)=J(k,\theta,t)$ ($k$-- real), we have
\begin{equation}
\sigma'(k^2,kt)=k\pi^{-1} \|J(k,\theta,t)\|^2_{L^2(\Sigma)}
\label{factor4}
\end{equation}

If $V(x)$ is only bounded, we can consider truncation $V_{R}(x)$
and the corresponding $\psi_R$ and $J_{R}(k,\theta,t)$. The last
vector-function is analytic in $\mathbb{C}^+$ and is continuous
down to the real line.

For any bounded $V$, we can introduce
\[
\mu(x,k,\xi)=\psi(x,k,\xi)\exp(-ik|x|)|x| \] Since $\psi\in
H^2(\mathbb{R}^3)$, the Sobolev embedding yields continuity of
$\psi$ and $\mu$.

 We start with
\begin{lemma}
For any compactly supported $V\in L^\infty(\mathbb{R}^3)$ and
$f\in L^2(\mathbb{R}^3)$, we have
\begin{equation}
\|J(k,\theta,\xi)\|_{L^2(\Sigma)}  \leq \left[ \sqrt{|k|} \Im
k\right]^{-1}\left[\|f(x)\|_2\|f(x) e^{2\Im k|x|}\|_2\right]^{1/2}
\label{always}
\end{equation}
\label{lemma-always}
\end{lemma}
\begin{proof}
For $\mu$,
\begin{equation}
-\Delta \mu -2\mu_r\left({ik}-\frac{1}{|x|}\right)+k\xi
V\mu=|x|\exp(-ik|x|)f(x) \label{sample}
\end{equation}
Divide the both sides by $2ik$, multiply by $\bar{\mu}(x)
|x|^{-2}$ and integrate over the spherical layer $l<|x|<L$. Taking
the real part, we have
\[
\frac{1}{L^2} \int_{|x|=L} |\mu(x,k,\xi)|^2d\sigma- \frac{1}{l^2}
\int_{|x|=l} |\mu(x,k,\xi)|^2d\sigma +\frac{\Im
k}{|k|^2}\int\limits_{l<|x|<L} \frac{|\nabla \mu|^2}{|x|^2}dx
\]
\begin{equation}
=-\Im\left[ \frac{1}{k} \int_{l<|x|<L}
\frac{f\bar{\mu}\exp(-ik|x|)}{|x|}dx\right]-\Re\left[
\frac{1}{ik}\int_{|x|=s} \frac{\mu'(x) \bar\mu (x)}{|x|^2}d\sigma
\Bigl. \Bigr|^{s=L}_{s=l}\right] \label{universal}
\end{equation}
Then, take $l\to 0, L\to\infty$ and use asymptotics at infinity
and regularity of $\mu$. We get
\[
\|J(k,\theta,\xi)\|_2^2+\frac{\Im k}{|k|^2}\int \frac{|\nabla
\mu|^2}{|x|^2}dx=-\Im\left[ \frac{1}{k} \int
\frac{f\bar{\mu}\exp(-ik|x|)}{|x|}dx\right]
\]
\[
=\Im \left[ \bar{k}^{-1} \int \psi \bar{f} e^{2\Im k|x|} dx
\right]=\frac{1}{|k|^2}\Im \left[ k \langle \psi(k),f e^{2\Im
k|x|}\rangle \right]=
\]
\begin{equation}
=\frac{1}{|k|^2}\Im \left[ k \langle \psi(k) e^{-ik|x|},f
e^{-ik|x|}\rangle \right]
 \label{balance}\end{equation}  The estimate
on $\|J\|$ now follows from (\ref{spec}).
\end{proof}
{\bf Remark.} Notice that the function $g(k)=k \langle
\psi(k),f\rangle$ is Herglotz in $\mathbb{C}^+$ and we have
factorization
\begin{equation}
\Im \left[ g(k)\right] =\|kJ(k,\theta,\xi)\|_2^2
\end{equation}
 for real $k$.
Also, for general $V$,  we have identity
\[
\frac{1}{|k|^2} \Im \left[ k\langle \psi(k),f\rangle\right]=(\Im
k)\left[\|\psi\|^2_2+\frac{1}{|k|^2}\|\nabla \psi\|_2^2\right]
\]

We will need the following auxiliary result.

\begin{lemma}
For any $V\in L^\infty (\mathbb{R}^3)$ and  $f\in
L^2(\mathbb{R}^3)$,
\begin{equation}
\|\psi_R(k)-\psi(k)\|_{L^2(\mathbb{R}^3)}\to 0 \label{first}
\end{equation}
\begin{equation}
\|\Delta \psi_R(k) -\Delta \psi(k)\|_{L^2(\mathbb{R}^3)}\to 0
\label{second}
\end{equation}
if $k\in \mathbb{C}^+$ is fixed.\label{approx}
\end{lemma}
\begin{proof}
We have
\[
\psi_R=\psi+k\xi P^{-1}_R(k)(V-V_R)\psi
\]
and
\[
\| \psi_R-\psi\|\leq \frac{|k\xi|}{(\Im k)^2} \|(V-V_R)\psi\|\to 0
\]
 Then
(\ref{second}) is an elementary corollary of (\ref{first}) and
equations
\[
-\Delta \psi+k\xi V\psi=k^2\psi+f,\, -\Delta \psi_R+k\xi
V_R\psi_R=k^2\psi_R+f
\]
\end{proof}
We will need some technical estimates
\begin{lemma}
For any $k\in \mathbb{C}^+$, $V\in L^\infty(\mathbb{R}^3)$, and
compactly supported $f\in L^2(\mathbb{R}^3$, we have
\begin{equation}
\int \frac{|\nabla \mu|^2}{|x|^2}\,dx\leq \frac{|k|}{[\Im
k]^3}\|f(x)\|_2 \|f(x) e^{2\Im k|x|}\|_2 \label{ineq1}
\end{equation}
\begin{equation}
\int\limits_{R<|x|<R+1} \frac{|
\mu|^2}{|x|^2}\,dx<\frac{C}{|k|[\Im k]^2}\left[1+\frac{1}{\Im
k}\right]\|f(x)\|_2 \|f(x)e^{2\Im k |x|}\|_2, \quad
R>1\label{ineq2}
\end{equation}

\begin{equation}
\int_\Sigma |\mu(r\sigma)|^2 d\sigma <C \frac{1+|k|\Im k}{[\Im
k]^4} \|f(x)\|_2 \|f(x)e^{2\Im k|x|}\|_2 \label{ineq3}
\end{equation}
where $C$'s are universal constants.
\end{lemma}
\begin{proof}
Consider $V_R$ obtained from $V$ by truncation. For the
corresponding $\mu_R$, (\ref{ineq1}) follows from (\ref{balance}).
Also, for any compact $K$ not containing zero, $\nabla\mu_R\to
\nabla\mu$ in $L^2(K)$ due to (\ref{second}) and we have
\[
\int_{K}\frac{|\nabla \mu|^2}{|x|^2}dx\leq \frac{|k|}{[\Im
k]^3}\|f(x)\|_2 \|f(x)e^{2\Im k|x|}\|_2
\]
Since $K$ is arbitrary, we have (\ref{ineq1}) for any bounded $V$.

To get (\ref{ineq2}), take $l=0$, $L=\rho$ in (\ref{universal}).
We have
\begin{eqnarray*}
\int\limits_\Sigma |\mu(\rho \sigma)|^2d\sigma+\frac{\Im k}{|k|^2}
\int\limits_{|x|<\rho} \frac{| \nabla \mu|^2}{|x|^2}\,dx=&\\
= -\Re\left[\frac{1}{ik} \int\limits_\Sigma {\mu'}(\rho\sigma)
\bar\mu(\rho\sigma) d\sigma\right]+\Im \left[ {\bar{k}}^{-1}
\int\limits_{|x|<\rho} \psi \bar{f} e^{2\Im k |x|}dx \right]
\end{eqnarray*}
Integrate in $\rho$ from $R$ to $R+1$ and use
\[
|\mu \mu'|\leq \frac 12 \left[\epsilon |\mu|^2+\epsilon^{-1}
|\mu'|^2\right]
\]
in the first term of the r.h.s. Then, taking $\epsilon=|k|/2$ and
 using
(\ref{ineq1}), we get (\ref{ineq2}). Since we have now the
estimates on $\mu(\rho \sigma)$ in $H^1_{\rm loc}(\mathbb{R}^+)$,
the standard Sobolev embedding argument yields (\ref{ineq3}).

\end{proof}
The last Lemma essentially says that the average decay of Green's
function $G(x,y,k)$ of $P(k)$ is {\bf always} at most $\exp(-|\Im
k|\cdot|x-y|)/|x-y|$. That fact gives strong improvement of
(\ref{ct}) and has no analogs in the spectral theory of
Schr\"odinger operators.

We will need the following standard result later on
\begin{lemma}
Consider $V(x)\in
L^\infty(\mathbb{R}^3)$,\,$V_{(m)}(x)=V(x)\chi_{|x|>m}, m>0$ and
the pencil $P_{(m)}(k)$ corresponds to potential $V_{(m)}(x)$.
Then for fixed $f\in L^2(\mathbb{R}^3)$ and $k\in \mathbb{C}^+$,
we have
\[
\psi_{(m)}=P_{(m)}^{-1}(k)f\to \psi=(-\Delta-k^2)^{-1}f,\quad
m\to\infty
\] uniformly over any compact in $\mathbb{R}^3$. \label{trunc}
\end{lemma}
\begin{proof}
The second resolvent identity reads
\[
\psi_{(m)}=\psi-k(-\Delta-k^2)^{-1}V_{(m)} \psi_{(m)}
\]
the last term can be written as
\[
k\int
\frac{e^{ik|x-y|}}{|x-y|}V_{(m)}(y)\psi_{(m)}(y)dy=k\int\limits_{|y|>m}
\frac{e^{ik|x-y|}}{|x-y|} V_{(m)}(y)\psi_{(m)}(y)dy
\]
The application of Cauchy-Schwarz inequality and (\ref{spec})
finishes the proof.
\end{proof}

 The next Lemma controls the
radial derivative of the solution in the case which is very close
to condition (\ref{l2}).

\begin{lemma}
Let $v(r)=\sup_{|x|=r}|V(x)|\in L^2(\mathbb{R}^+)$. Then, for any
fixed $k\in \mathbb{C}^+$ and any $f\in L^2(\mathbb{R}^3)$
supported within $|x|<\rho$, we have
\begin{equation}
4\Im k\int\limits_{|x|>\rho} \frac{|\mu'(x,k)|^2}{|x|^2}dx
<C(k)\int\limits_\rho^\infty v^2(r)dr+\int_\Sigma |\nabla_\tau
\mu(\rho \sigma)|^2 d\sigma \label{rad}
\end{equation}\label{lemma-l2}
\end{lemma}
\begin{proof}
In the spherical coordinates, the equation for $\mu$ reads as
follows
\begin{equation}
-\mu''-2ik\mu'-\frac{B}{r^2}\mu +k\xi V\mu=r\exp(-ikr)f(r\sigma),
\, r>0,\sigma\in \Sigma \label{spherc}
\end{equation}

Instead of $V$, consider $V_R$ and the corresponding $\mu_R$.
Multiply the both sides by $\mu'_R$ from the right and integrate
from $\rho$ to infinity. Taking the real part yields

\[
\|\mu_R'(\rho,\theta)\|^2+4\Im k\int\limits_\rho^\infty
\|\mu_R'(r,\theta)\|^2dr+2\int_{|x|>\rho} \frac{|\nabla_\tau
\mu_{R}(x)|^2}{|x|^3}=
\]

\begin{equation}
=-\rho^{-2}\langle B\mu_R(\rho,\theta), \mu_R(\rho, \theta)\rangle
 -2\xi
\Re\left[k\int\limits_\rho^\infty \langle V_R\mu_R, \mu_R'\rangle
dr\right] \label{sample2}
\end{equation}
The second integral in the r.h.s. can be bounded by $\displaystyle
C(k)\int_R^\infty v^2(r)dr$ due to (\ref{ineq1}) and
(\ref{ineq3}). Thus, we have the statement of the Lemma for each
$R$. Take $R\to\infty$. Lemma \ref{approx} and Sobolev embedding
theorem, allows one to go to the limit and get (\ref{rad}).
\end{proof}

Now, we have enough information to prove Theorems \ref{theorem2}
and \ref{theorem3}. The main idea is the same as in the proof of
Theorem \ref{theorem1}. That is to use subharmonicity in $k\in
\mathbb{C}^+$ of the function $\ln \|J_R(\theta,k,\xi)\|$ to
obtain the lower bound on the entropy
\[
\int\limits_a^b d\lambda \int\limits_{-T_1}^{T_1} \ln
\sigma'(\lambda,t) d\lambda
\]
by using factorization (\ref{factor4}). To do that, we have the
uniform bound from above given by (\ref{always}). This bound is
true always, regardless of the behavior of potential at infinity.
The only thing we need to do to make the argument work is to
provide a bound from below for $\ln \|J_R(\theta,k,\xi)\|$ which
would be uniform in $R$. Moreover, it is enough to prove this
bound for at least some point $k=k_0$ inside a triangle considered
in the proof of Theorem \ref{theorem1}. Getting this bound will
involve the information on decay of $V$ and will be the core of
the proofs for the next two Theorems.

{\it Proof of Theorem \ref{theorem2}.}

  Since the a.c. part of the measure
is invariant under the trace-class perturbations (e.g.,
\cite{rs3}, Birman-Kuroda Theorem), it is enough to prove the
statement for $V_m=div L_m, L_m(x)= L(x)\cdot b_m(|x|)$ where $m$
is arbitrary fixed number and $b_m(t)=0$ on $[0,m-1]$, $b_m(t)=1$
for $t>m$ and is smooth on $[m,m+1]$. For any $f\in
L^2(\mathbb{R}^3)$, $d\sigma_m(\lambda,t)$ denotes the spectral
measure of $f$ corresponding to $H_m(t)$. We will be taking $m$
large later on.

 Assume that the support of $f$ is within $|x|<1$ and consider
compactly supported potentials $V_{m,R}=div\, L_{m,R}$ with
$L_{m,R}(x)=L(x)\cdot b_{m,R}(|x|)$, where smooth $b_{m,R}$ is
such that $b_{m,R}(t)=0$ for $t\in (0,m-1)\cup (R+1,\infty)$,
$b_{m,R}(t)=1$ for $t\in (m,R)$.

 For that $V_{m,R}$, multiply the both sides of (\ref{sample})
by $|x|^{-2}$, integrate  over $|x|>\rho>1$, and use asymptotics
at infinity.  We then have

\[
 \int\limits_\Sigma J_{m,R}(\theta, k,\xi) d\theta=\hspace{5cm}
 \]
 \begin{equation}
\frac{1}{\rho^2}\int\limits_{|x|=\rho} \mu_{m,R}(x,
k,\xi)d\sigma -\frac{i}{2k\rho^2}\int\limits_{|x|=\rho}
\mu'_{m,R}(x, k,\xi)dx+\frac{\xi}{2i} \int\limits_{|x|>\rho}
\frac{V_{m, R}\mu_{m,R}}{|x|^2}dx \label{form1}
\end{equation}
The last integral is equal to
\[
-\int \frac{ L_{m,R} \cdot \nabla \mu_{m,R}}{|x|^2}dx+2\int
\frac{2\mu_{m,R} }{|x|^3}  \left[ L_{m,R}\cdot
\frac{x}{|x|}\right] dx
\]
and its absolute value is not greater than
\[
C(k)\int \frac{|L_{m,R}(x)|^2}{|x|^2+1} dx
\]
by Cauchy-Schwarz and (\ref{ineq1}), (\ref{ineq2}). Notice that
for fixed $k$ the last quantity can be made arbitrarily small
uniformly in $R$ by choosing $m$ large.

 Now, fix any $k_0\in
\mathbb{C}^+$. Take $f$ to be spherically symmetric nonnegative
and with unit norm. Then, by Lemma \ref{trunc}
\[
\mu_{m,R}(x,k,\xi)\to \mu^0(x,k)=\int
\frac{|x|e^{-ik(|x|-|x-y|)}}{|x-y|}f(y)dy, \, m\to\infty
\]
for fixed $k\in \mathbb{C}^+$ uniformly in $R>R_0$ and in $x\in K$
for any compact $K$. By Lemma~\ref{trunc} and the theorem on trace
of $H^2$ functions, we have
\[
\mu'_{m,R}\to (\mu^0)', m\to\infty
\]
in $L^2(S_r)$ on any fixed sphere $S_r=\{|x|=r\}$.

This $\mu^0$ is spherically symmetric since $f$ is spherically
symmetric. Moreover, for $|x|\to\infty$,
\[
\mu^0(x,k)\to A^0(k)=\int e^{-ik\langle \theta, y\rangle}
f(|y|)dy=\int\limits_0^\infty tf(t)\frac{\sin (kt)}{k}dt
\]
and
\[
(\mu^0)'(x,k)\to 0
\]
where $A^0(k)$ is  the amplitude of $f$ with respect to
unperturbed operator. This function is entire and therefore has
only finite number of zeroes in any compact in $\mathbb{C}$. For
any $k_0\in \mathbb{C}^+$ which is not zero, we can arrange first
$\rho$ and then $m$ such that the difference in the r.h.s. of
(\ref{form1}) has absolute value greater than some $2\delta$ and
the last term in the r.h.s. of (\ref{form1}) has absolute value
smaller than $\delta$ (all that uniformly in $R>R_0$ and $\xi\in
[-T_1, T_1]$, where $T_1$ is any fixed constant). We get
\[
\left| \frac{1}{\rho^2} \int\limits_{|x|=\rho} \mu_{m, R}(x,
k_0,\xi)d\sigma-\frac{i}{2k_0\rho^2}\int\limits_{|x|=\rho}
\mu'_{m, R}(x, k_0,\xi)d\sigma+\frac{\xi}{2i} \int
\frac{V_{m,R}\mu_{m, R}}{|x|^2}dx\right|>\delta>0
\]
for any $|\xi|<T_1$ and any $R>R_0$. Thus, for $k=k_0$, we have
\begin{equation}
\|J_{m,R}(\theta,k_0,\xi)\|\geq |\langle J_{m,R}(\theta,
k_0,\xi),1\rangle|>\delta \label{est-bel}
\end{equation}
Then, since we also have the estimate (\ref{always}) and
factorization (\ref{factor4}), the proof of absolute continuity
for the measure repeats the argument for Theorem \ref{theorem1}.
The logic here is that we first choose an interval for the
spectral parameter and for coupling constant, then take $f$ and
some $k_0$ which is inside the triangle and is not a zero of
$A^0(k)$. Then we find large $\rho$ and after that make a
truncation by $m$ so that  the uniform in $R$ estimates
(\ref{est-bel}) hold.

Now, we have two identities
\[
\sigma'_{m,R}(k^2,kt)=k\pi^{-1}\|J_{m,R}(k,\theta,t)\|_2^2
\]
and
\[
\Im (k\langle \psi_{m,R}(k),f\rangle)=\|kJ_{m,R}(k,\theta,t)\|_2^2
\]
The functions $k\langle \psi_{m,R}(k),f\rangle$ are Herglotz in
$\mathbb{C}^+$ having uniform in $m,R$ estimates. That yields the
uniform bound on variations, i.e.
\[
\int_J \sigma'_{m,R}(k^2,kt)dk<C(J,T_1)
\]
uniformly in $m,R,|t|<T_1$. Here $J$ is any interval in
$\mathbb{R}$.

Now the repetition of subharmonicity argument gives the uniform
bounds on the entropy

\[
\int\limits_{a^2}^{b^2}\int\limits_{-T}^{T} \ln^-
\sigma'_{m,R}(\lambda,t)d\lambda dt>C
\]

 The weak-star convergence of
$d\sigma_{m,R}(\lambda,t)$ to $d\sigma_m (\lambda,t)$ as
$R\to\infty$ is a simple corollary of Lemma \ref{approx}. It
allows to conclude that
\[
\int\limits_{a^2}^{b^2}\int\limits_{-T}^{T} \ln^-
\sigma'_m(\lambda,t)d\lambda dt>-\infty
\]
thus  $(a^2,b^2)\subseteq \sigma_{ac}(H_m(t))$ for generic $t\in
[-T,T]$. Recall that parameter $m$ corresponds to cutting $L$ off
inside the ball of radius $m$. As mentioned earlier, this
subscript $m$ can be dropped due to trace-class type argument.
\quad $\Box$

{\bf Remark.} This result suggests that the method of \cite{dd1}
can probably be pushed forward to prove Theorem \ref{theorem1} for
any coupling constant. Also, in the proof we have control over
$\langle J,1\rangle$ and that implies that nontrivial energy is
always present on low angular modes. We do not think that that is
the case when $V$ decays without substantial oscillation.

{\it Proof of Theorem \ref{theorem3}.}

By Weyl's Theorem on essential spectrum \cite{rs},
$\sigma_{ess}(H(t))=\mathbb{R}^+$. Consider $V_{m,R}(x)=V(x) \cdot
\chi_{m<|x|<R}, \, 1<m<R$ and assume that $f$ is spherically
symmetric, has support within the unit ball and has the unit norm.
Multiply (\ref{spherc}) from the right by $r\mu'_R$ and integrate
from $\rho$ to infinity. Similarly to Lemma \ref{lemma-l2}, we
have

\[
\rho\|\mu_{m,R}'(\rho)\|^2+\int\limits_\rho^\infty
\|\mu_{m,R}'(r)\|^2dr+4\Im k\int\limits_\rho^\infty
r\|\mu_{m,R}'(r)\|^2dr+\int_{|x|>\rho}\frac{|\nabla_\tau
\mu_{m,R}(x)|^2}{|x|^2}dx\leq
\]
\[
\frac{1}{\rho}\int\limits_{|x|=\rho} |\nabla_\tau \mu_{m,R}(x)|^2
d\sigma+2\xi \left|k \int\limits_{|x|>\rho} \frac{V_{m,R}\mu_{m,R}
\bar{\mu}_{m,R}'}{|x|}dx \right|
\]
The last integral is bounded by
\[
C(k)\left[\int\limits_\rho^\infty
rv_m^2(r)dr\right]^{1/2}\left[\int\limits_\rho^\infty
r\|\mu_{m,R}'(r)\|^2dr\right]^{1/2}
\]
due to (\ref{ineq1}) and (\ref{ineq3}). Using inequality $2ab\leq
\epsilon a^2+\epsilon^{-1}b^2$ for the last product, we get
\[
\int_{|x|>\rho}\frac{|\nabla_\tau \mu_{m,R}(x)|^2}{|x|^2}dx+2\Im
k\int\limits_\rho^\infty r\|\mu_{m,R}'(r)\|^2dr<
\]
\[
<C(k)\left[ \int\limits_\rho^\infty rv^2_m(r)dr
+\frac{1}{\rho}\int\limits_{|x|=\rho} |\nabla_\tau \mu_{m,R}(x)|^2
d\sigma\right]
\]
Notice that we infact show the weighted $L^2$ estimate for the
full gradient
\begin{equation}
\int_{|x|>\rho}\frac{|\nabla
\mu_{m,R}(x)|^2}{|x|^2}dx<C(k,\xi)\left[ \int\limits_\rho^\infty
rv^2_m(r)dr +\frac{1}{\rho}\int\limits_{|x|=\rho} |\nabla_{\tau}
\mu_{m,R}(x)|^2 d\sigma\right] \label{gradient}
\end{equation}

Now fix any positive interval for the spectral parameter and let
the coupling constant $\xi\in [-T_1,T_1]$.

 In (\ref{universal}), let $L\to\infty$. If $l=\rho>1$, the
 integral with $f$ will be disappear and
\[
\|J_{m,R}(k,\theta,\xi)\|_2^2 =\hspace{5cm}
\]
\begin{equation}
=\frac{1}{\rho^2} \int\limits_{|x|=\rho}
|\mu_{m,R}(x,k,\xi)|^2d\sigma -\frac{\Im
k}{|k|^2}\int\limits_{\rho<|x|} \frac{|\nabla
\mu_{m,R}|^2}{|x|^2}dx +\Re\left[
\frac{1}{ik}\int\limits_{|x|=\rho} \frac{\mu_{m,R}'(x)
\bar\mu_{m,R} (x)}{|x|^2}d\sigma \right]
\end{equation}
Now, as $m\to \infty$, the first term in the r.h.s. approaches

\begin{equation}
\frac{1}{\rho^2} \int\limits_{|x|=\rho} |\mu^0(x,k,\xi)|^2d\sigma
\label{expression}
\end{equation}
 uniformly in $R>m$ and in $k\in D$ where $D$ is any domain in $\mathbb{C}^+$.
Now, if $\rho\to\infty$, (\ref{expression}) will converge to
$|A^0(k)|^2$. We take any $k=k_0\in \mathbb{C}^+$ which is not a
zero of $A^0(k)$. If $k_0$ is fixed,

\[
\frac{1}{\rho^2} \int\limits_{|x|=\rho}
|\mu_{m,R}(x,k_0,\xi)|^2d\sigma>2\delta
\]
if we first choose $\rho$ and then $m$ to be large. This
inequality holds uniformly in $R$.

Consider
\[
-\frac{\Im k_0}{|k_0|^2}\int\limits_{\rho<|x|} \frac{|\nabla
\mu_{m,R}|^2}{|x|^2}dx +\Re\left[
\frac{1}{ik_0}\int\limits_{|x|=\rho} \frac{\mu_{m,R}'(x)
\bar\mu_{m,R} (x)}{|x|^2}d\sigma \right]
\]
The function $\mu^0$ is spherically symmetric and therefore
$\nabla_\tau \mu^0=0$. Thus, by (\ref{gradient}),
\[
\left|\frac{\Im k_0}{|k_0|^2}\int\limits_{\rho<|x|} \frac{|\nabla
\mu_{m,R}|^2}{|x|^2}dx\right|<\delta
\]
uniformly in $R>R_0$, if $k_0$ and $\rho$ are fixed and $m$ is
large. Moreover, we have
\[
\left|\Re\left[ \frac{1}{ik_0}\int\limits_{|x|=\rho}
\frac{\mu_{m,R}'(x) \bar\mu_{m,R} (x)}{|x|^2}d\sigma
\right]\right|<\delta
\]
uniformly in $R>R_0$ if we first choose $\rho$ and then $m$ to be
large. This is due to the fact that
$\partial\mu^0(x,k_0,\xi)/\partial r \to 0$ as $|x|\to \infty$.

Thus, after arranging  $\rho$ and  $m$, we finally have
\[
\|J_{m,R}(k_0,\theta,\xi)\|_2^2>\delta
\]
uniformly in $R>R_0$. This uniform in  $R$ estimate from below
allows us to repeat the arguments from the Theorem \ref{theorem1}
and finish the proof. $\Box$

{\bf Remark.} Notice that even in one-dimensional case, the
condition considered in the last Theorem leads to WKB correction
\[
\exp\left[ \frac{1}{2ik}\int\limits_0^r V(s)ds \right]
\]
in the asymptotics of the Green function. The main point in the
last proof is to show that the complete gradient of $\mu$ is small
(not only its radial component). If so, the general identity
(\ref{balance}) provides the bound from below for $\|J\|$. It is
important to mention that the usual $L^2$ decay of potential
guarantees that $\mu'$ is small (due to Lemma \ref{lemma-l2}) but
we can say nothing about the size of $\nabla_\tau\mu$.

Now, we further study the Schr\"odinger pencil. As we saw before,
the main equation to consider is (\ref{spherc}), which can be
rewritten as
\begin{equation}
\mu'=\kappa \frac{B}{r^2}\mu +\frac{\xi}{2i} V\mu+w_1+w_2, \quad
r>0
\end{equation}
with
\[
\kappa= -\frac{1}{2ik}, \quad w_1=-\frac{1}{2ik}\mu'', \quad
w_2=-\frac{re^{-ikr}}{2ik}f
\]
This is not quite an evolution equation on $L^2(\mathbb{R}^+,
L^2(\Sigma))$ because of the second derivative present, but we can
study the asymptotics of solution by writing Duhamel formula
\begin{equation}
\mu(r)=U(\rho, r,k)\mu(\rho)+\int_\rho^r U(s,r,k)w_1(s)ds
\label{duhamel}
\end{equation}
where $\rho>1$ and $f$ is supported on $[0,1]$ and
\begin{equation}
U'(\rho,r,k)=\kappa \frac{B}{r^2}U(\rho,r,k) +\frac{\xi}{2i}
VU(\rho,r,k),\, U(\rho,\rho,k)=I \label{evol1}
\end{equation}
Now, by considering $V_{(m)}=V\cdot \chi_{|x|>m}$ instead of $V$
and taking $f$ spherically symmetric, we can always make sure that
$\mu(\rho)$ in (\ref{duhamel}) is close to $\mu^0(\rho)$, a
constant function in angles, in the uniform norm.

Now, we know  that $\mu'$ has small $L^2$ norm provided $V$
satisfies conditions of Lemma \ref{lemma-l2}. Integration by parts
and rather simple estimates on $\partial_s U(s,r,k)$ allow one to
estimate the second term in (\ref{duhamel}).
Therefore, to show that $\|\mu(r)\|$ is bounded away from zero, we
need to concentrate mostly on the first term $U(\rho,
r,k)\mu(\rho)$. Notice that
\[
\Re\left[ \kappa \frac{B}{r^2}+\frac{\xi}{2i} V\right]=\frac{\Im
k}{2|k|^2} \cdot \frac{B}{r^2}
\]
Since $B$ is nonpositive, the norm  $\|U(\rho,r,k)\eta\|$
decreases in $r$. Moreover, it might be that hight oscillation of
$V$ kicks the Fourier spectrum of $u(r)=U(\rho,r,k)\eta$ to the
higher and higher modes where the energy is dissipated due to the
presence of $B$. In other words, we do not have any proof that
$\|u(r)\|$ does not go to zero even for $V$:
$|V(r,\theta)|<Cr^{-1+}$. Moreover, it might very well be that
$\|u(r)\|$ does go to zero for some choice of $V$ satisfying this
condition. Therefore, we have to use the following modification of
the Hamiltonian $H$ itself. As we know, the operator $H$ is
unitarily equivalent to the operator
\[
-\frac{d^2}{dr^2}-\frac{B}{r^2}+V(r)
\]
defined on $L^2(\mathbb{R}^+,L^2(\Sigma))$ with Dirichlet boundary
condition at zero. In the previous part of the paper, we
introduced the coupling constant in front of the potential. Now,
we consider a different family of operators. Let
$\lambda_m=-m(m+1), m=0,1,\ldots $ be distinct eigenvalues of $B$
and $Y_{l}^m$-- the corresponding spherical harmonics ($|l|\leq
m$). Let $\alpha\in (0,1)$ be some positive parameter to be chosen
later and $r_m$ be the points of intersection of the graph of
$\omega(r)=r^{\alpha}$ with levels $|\lambda_m|^{1/2},\,
m=0,1,\ldots$. On $r>1$, we introduce the function $s(r)$ which is
piecewise constant and equals to $|\lambda_m|^{1/2}$ on each
$I_m=[r_m,r_{m+1})$.

 We consider the function
$s_1(\omega,r)$ defined for $\omega\geq 0, r>1$ such that
$s_1(\omega,r)=0$ for $\omega> s(r)$ and $s_1(\omega,r)=1$ for
$\omega\leq s(r)$. The decomposition of unity on $r>1, \omega\geq
0$ is defined through $s_2(\omega,r)=1-s_1(\omega,r)$. For each
$r>1$, these $s_{1(2)}$ define the multipliers and the
corresponding operators

\[
M_{1(2)}(r)f=\sum\limits_{l,m} Y_l^m f_l^m
s_{1(2)}(|\lambda_m|^{1/2},r)
\]
where $f\in L^2(\Sigma)$ and $f_l^m$ are Fourier coefficients with
respect to spherical harmonics. The point here is that we want to
separate frequencies along the level $\omega\sim r^\alpha$ and
define
\[
B_{1(2)}(r)=B M_{1(2)}(r), \quad \widetilde
H(t)=-\frac{d^2}{dr^2}+t\left[-\frac{B_1(r)}{r^2}+V(r)\right]-\frac{B_2(r)}{r^2},
\quad t\in \mathbb{R}
\]
For $r\in [0,1)$, we let $\widetilde{H}(t)=H(t)$, this interval is
not important. Of course, $\widetilde H(1)=H(1)$. The operator
$\widetilde H(t)$ can be rewritten as
\[
\widetilde H(t)=H(t)+(1-t)\frac{B_1(r)}{r^2}
\]
Notice that $B_1(r)r^{-2}$ is bounded in the Hilbert space
$L^2([1,\infty), L^2(\Sigma))$. Therefore for self-adjoint
$\widetilde{H}(t)$, we have
$\cal{D}(\widetilde{H}(t))=H^2(\mathbb{R}^3)$ provided that $V\in
L^\infty(\mathbb{R}^3)$. Essentially, in this approach we treat
\[
\widetilde{V}(r)=-\frac{B_1(r)}{r^2}+V(r)
\]
as the perturbation of
\[
\widetilde{H}^0=-\frac{d^2}{dr^2}-\frac{B_2(r)}{r^2}
\]
The operator $\widetilde{H}^0$ can be easily decoupled into the
orthogonal sum of one-dimensional Schr\"odinger operators with
explicit potentials. It is an easy exercise then to check that the
spectrum of $\widetilde{H}^0$ is $[0,\infty)$ and is purely a.c.
One has to note thought that perturbation $\widetilde{V}$ is not a
multiplication by a function any longer.

 Now, we are
ready to formulate our result.
\begin{theorem}
 If $\alpha=2/3-$  and $|V(x)|<C\langle
x\rangle^{-\gamma},\,\gamma>3/2-\alpha $, then
$\sigma_{ac}(\widetilde{H}(t))=\mathbb{R}^+$ for generic~$t$.
\label{theorem-last}
\end{theorem}

{\bf Remark.} Since the a.c. spectrum covers the positive
half-line for generic $t$, it is true for some $t$ accumulating to
$1$. That suggests (but does not prove) that the a.c. spectrum is
likely to be preserved for $t=1$ (i.e. for the original
Schr\"odinger operator), at least under the $5/6+$ assumption on
decay. In any case, this result is the first one when we are able
to go below $1$ in the decay assumption on the potential.

The proof follows the same lines. Consider truncations in space
\[
\widetilde{V}_R(r)= \widetilde{V}(r) \chi_{r<R}
\]
and damping of $B_2$ as
\[
\widetilde{H}_{R,b}=-\frac{d^2}{dr^2}+\frac{B_{2,b}(r)}{r^2}
+\widetilde{V}_R
\]
where $B_{2,b}(r)=B_b M_2(r)$,
\[
B_bf=\sum\limits_{l,m, |m|<b} Y_l^m \lambda_m f_l^m-b(b+1)
\sum\limits_{l,m, |m|\geq b} Y_l^m  f_l^m
\]
Here $b>R^\alpha$ and the damping is introduced to reduce the
problem to one-dimensional Schr\"odinger operator with bounded
operator-valued potential whose norm is in $L^1[1,\infty)$. For
these operators, we know absorption principle, absence of embedded
positive eigenvalues, etc. The point, though, is to prove estimate
on the entropy (e.g., (\ref{entropy1})) which is uniform in $b$
and $R$. Then, the following simple approximation result will do
the job.

\begin{lemma}
For any $f(r)\in L^2(\mathbb{R}^+,L^2(\Sigma))$ and any $z\in
\mathbb{C}^+$, we have
\[
\langle (\widetilde{H}_{R,b}-z)^{-1}f,f\rangle \to\langle
(\widetilde{H}-z)^{-1}f,f\rangle
\]
as $R\to\infty, b\to\infty$.\label{weak-c}
\end{lemma}
\begin{proof}
The second resolvent identity yields
\[\langle R_{R,b}(z)f,f\rangle=\langle
R(z)f,f\rangle-\langle
\left(\frac{B_{2,b}(r)-B_2(r)}{r^2}\right)R(z)f+(\widetilde{V}_R-\widetilde{V})R(z)f,R^*_{R,b}(z)f\rangle
\]
Since
\[
\left\|\frac{B_{2,b}(r)-B_2(r)}{r^2} g\right\|\to 0,\,
\left\|(\widetilde{V}_R-\widetilde{V})g \right\|\to 0\quad
R\to\infty, b\to\infty
\]
for fixed $g\in \cal{D}(\widetilde{H})=H^2(\mathbb{R}^3)$, we have
the statement of a Lemma.
\end{proof}
This  Lemma yields the weak--star convergence of the spectral
measures $d\sigma_{R,b}(\lambda)$ to $d\sigma(\lambda)$, where the
spectral measures are calculated for fixed $f$.

For $\widetilde{H}_{R,b}$, the analog of Lemma \ref{lemmanew} (and
Lemma \ref{lemma-factor}) holds true.
\begin{lemma}
For any $f(r)\in L^2(\mathbb{R}^+,L^2(\Sigma))$ with compact
support, we have
\[
\left[(\widetilde{H}_{R,b}-k^2-i(+0))^{-1}f\right](r)\sim
\exp(ikr)A_{R,b}
\]
as $r\to\infty$. Moreover, for the spectral measure of $f$, we
have
\[
\sigma'_{R,b}(k^2)=k\pi^{-1}\|A_{R,b}(k)\|^2, k>0
\]
\end{lemma}

Just like in the previous sections, we can not say anything about
the asymptotics of the Green function for $\widetilde{H}$.
Therefore, we introduce the coupling constant against
$\widetilde{V}$ and consider
 the associated quadratic pencils
\[
\widetilde{P}(k,\xi)=\widetilde{H}^0+k\xi\widetilde{V}-k^2,\,\widetilde{P}_{R,b}(k,\xi)=\widetilde{H}^0_b+k\xi\widetilde{V}_R-k^2
\]
They are also hyperbolic and we have Lemma \ref{trifle}. For any
compactly supported $f(r)\in L^2(\mathbb{R}^+,L^2(\Sigma))$, we
introduce $\psi_{R,b}=\widetilde{P}_{R,b}^{-1}(k)f$,
$\mu_{R,b}=\exp(-ikr)\psi_{R,b}$,
$J_{R,b}(k,\xi)=\lim_{r\to\infty} \mu_{R,b}(r,k, \xi)$. We have
\begin{equation}
\sigma'_{R,b}(k^2,kt)=k\pi^{-1}\|J_{R,b}(k,t)\|^2
\label{another-fact}
\end{equation}

and the following analog of Lemma~\ref{lemma-always}.

\begin{lemma}
For any compactly supported $f\in L^2(\mathbb{R}^+,L^2(\Sigma))$
and $k\in \mathbb{C}^+$, we have
\begin{equation}
\|J_{R,b}(k,\xi)\|_{L^2(\Sigma)}  \leq \left[ \sqrt{|k|} \Im
k\right]^{-1}\left[\|f(r)\|_2\|f(r) e^{2\Im k|r|}\|_2\right]^{1/2}
\label{always1}
\end{equation}
uniformly in $R>R_0, b>1$.
\end{lemma}
The estimate on the derivative of $\mu$ can be obtained in the
same way.

\begin{lemma}
For any $k\in \mathbb{C}^+$, we have
\begin{equation}
\int\limits_0^\infty \|\mu'_{R,b}(r,k)\|^2dr<C(k) \label{rad-der}
\end{equation}
\end{lemma}

 What makes the situation different is the
behavior of evolution $\widetilde{U}(\rho,r)$

\begin{equation}
\widetilde{U}'(\rho,r,k)=\kappa
\frac{B_2(r)}{r^2}\widetilde{U}(\rho,r,k) +\frac{\xi}{2i}
\widetilde{V}\widetilde{U}(\rho,r,k),\,
\widetilde{U}(\rho,\rho,k)=I
\end{equation}
as $r\to\infty$. Recall that
$\widetilde{V}_{(d)}=\widetilde{V}\cdot \chi_{r>d}$. We have
\begin{lemma}
Fix $k\in \mathbb{C}^+$ and let $\alpha=2/3-$,
$\gamma>3/2-\alpha$, $|\xi|<T_1$. Assume that $|V(x)|\leq C\langle
x\rangle^\gamma$ and consider the evolution
\begin{equation}
\widetilde{U}'(\rho,r,k)=\kappa
\frac{B_2(r)}{r^2}\widetilde{U}(\rho,r,k) +\frac{\xi}{2i}
\widetilde{V}_{(d)}\widetilde{U}(\rho,r,k),\,
\widetilde{U}(\rho,\rho,k)=I
\end{equation}
where $d(k,V,T_1)$ is large enough. Then, we have
\begin{equation}
\liminf_{r\to\infty} \|\widetilde{U}(1,r,k)1
\|>\delta(k,\gamma,T_1)>0 \label{bound-below}
\end{equation}
Assume also that $\widetilde{V}$ has compact support in $[0,R]$.
Then, for each $\eta\in L^2(\Sigma)$,
\begin{equation}
\int\limits_{t}^{\infty} \|\partial_\rho
\widetilde{U}^*(\rho,\infty)\eta\| ^2d\rho= \bar{o}(1)\cdot
\|\eta\|^2 \label{bound-der}
\end{equation}
as $t\to\infty$ uniformly in $R$.

The same results hold true for the case when truncation by $R$ and
damping by $b$ are introduced. The resulting estimates are uniform
in $R>R_0, b>R^\alpha$.
 \label{twist}
\end{lemma}
\begin{proof}
For simplicity, we take $ \xi=-2$ and $k=i/2$. Then, we have
\begin{equation}
u'=\frac{B_2(r)}{r^2}u +i \widetilde{V}_{(d)}u,\, u(1)=1
\end{equation}
Let us study this evolution. Obviously, $\|u\|$ decreases. We
split $u(r)=M_1(r)u+M_2(r)u=u_1(r)+u_2(r)$.  Operators
$M_{1(2)}=const$ on the intervals $I_m=[r_m, r_{m+1})$ and act as
orthoprojectors, also $r_m\sim m^{1/\alpha}$, $|I_m|\sim
m^{1/\alpha-1}$. Let us control the variation of $\|u_{1(2)}\|$ on
each of $I_m$. We have
\[
\frac{d}{dr} \left[
\begin{array}{c}
u_1(r)\\
 u_2(r)
 \end{array}
 \right]=\hspace{5cm}
 \]
\[
=\left[
\begin{array}{cc}
iB_1(r)r^{-2}+iV^{11}(r) & iV^{12}(r) \\
 iV^{21}(r)  & B_2(r)r^{-2}+iV^{22}(r)
 \end{array}
 \right] \left[
\begin{array}{c}
u_1(r)\\
 u_2(r)
 \end{array}
 \right]
\]
where
\[
V^{ij}=M_iV_{(d)} M_j, \quad\left[
\begin{array}{c}
u_1(r_m+0)\\
 u_2(r_m+0)
 \end{array}
 \right]=\left[
\begin{array}{c}
\alpha_m\\
 \beta_m
 \end{array}
 \right]
\]
Consider $U_{1(2)}$ acting on Ran$M_{1(2)}|_{I_m}$ and  defined as
follows
\[
U_1'(\rho,r)=\left[iB_1(r)r^{-2}+iV^{11}(r)\right]U_1(\rho,r),
U_1(\rho,\rho)=I
\]
\[
U_2'(\rho,r)=\left[B_2(r)r^{-2}+iV^{22}(r)\right]U_2(\rho,r),
U_2(\rho,\rho)=I
\]
where $r_m<\rho<r<r_{m+1}$. $U_1$ is unitary and $U_2$ is a
contraction satisfying
\[
\|U_2(\rho,r)\|\leq \exp\left[ -\lambda_m
\frac{r-\rho}{r\rho}\right],\, r_m<\rho<r<r_{m+1}
\]
on Ran$M_2$. The dynamics under this evolution is as follows:
$U_1$ doesn't change the norm of $u_1$, $U_2$ suppresses $u_2$,
and interaction between $u_1$ and $u_2$ is small due to decay of
$V^{12}$. This situation is standard in asymptotical analysis.

By Duhamel,
\[
u_1(r)=U_1(r_m,r)\alpha_m +i \int\limits_{r_m}^r
U_1(\rho,r)V^{12}(\rho)u_2(\rho)d\rho
\]

\[
u_2(r)=U_2(r_m,r)\beta_m +i \int\limits_{r_m}^r
U_2(\rho,r)V^{21}(\rho)u_1(\rho) d\rho
\]
When moving from $I_m$ to $I_{m+1}$ the dimension of Ran$M_{1(2)}$
increases/decreases by the geometric multiplicity of
$\lambda_{m+1}$. Therefore,
\[
\|u_2(r)\|\leq
\|\beta_m\|\exp\left[-\lambda_m\frac{r-r_m}{rr_m}\right]+c\int\limits_{r_m}^r
\rho^{-\gamma}\exp\left[-\lambda_m \frac{r-\rho}{\rho
r}\right]d\rho
\]
and
\[
\|\beta_{m+1}\|\leq
\exp(-Cm^{1-\alpha^{-1}})\|\beta_m\|+c\int\limits_{r_m}^{r_{m+1}}
\rho^{-\gamma}\exp\left[-\lambda_m \frac{r_{m+1}-\rho}{\rho
r_{m+1}}\right]d\rho<
\]
\[
\leq
\exp(-Cm^{1-\alpha^{-1}})\|\beta_m\|+cm^{\alpha^{-1}(1-\gamma)-1}
\]
Let $\kappa_1= 1-\alpha^{-1}$,
$\kappa_2=-\alpha^{-1}(1-\gamma)+1$. The simple iteration gives
\[
\|\beta_m\|<C\sum_{j=1}^m \exp\left[
-C(m^{\kappa_1+1}-j^{\kappa_1+1})\right]j^{-\kappa_2}<Cm^{-\kappa_1-\kappa_2}
\]

 For $\alpha_{m+1}$,
\[
\|\alpha_{m+1}\|\geq \|\alpha_m\|-c\|\beta_m\|
\int\limits_{r_m}^{r_{m+1}} \rho^{-\gamma}\exp\left[-\lambda_m
\frac{\rho-r_m}{\rho r_m}\right]d\rho
\]
\[
-c\int\limits_{r_m}^{r_{m+1}} \rho^{-\gamma}
\int\limits_{r_m}^\rho s^{-\gamma} \exp\left[-\lambda_m
\frac{\rho-s}{\rho s}\right] ds d\rho
\]

For
\[
\zeta_m=\int\limits_{r_m}^{r_{m+1}} \rho^{-\gamma}
\exp\left[-\lambda_m \frac{\rho-r_m}{\rho r_m}\right]d\rho, \quad
|\zeta_m|\leq Cm^{\alpha^{-1}(1-\gamma)-1}
\]
and
\[
\eta_m=\left|\int\limits_{r_m}^{r_{m+1}}\rho^{-\gamma}\int\limits_{r_m}^{\rho}
s^{-\gamma}\exp\left[-\lambda_m \frac{\rho-s}{s \rho}\right]ds
d\rho\right|<Cm^{2(1-\gamma)/\alpha-2}
\]
If $\alpha=2/3-$ and $\gamma>3/2-\alpha$, then

\[
\|\alpha_{m+1}\|\geq \|\alpha_m\|-cm^{-1-}
\]
Taking $d$ large and taking different $\gamma'\in
(3/2-\alpha,\gamma)$ in all estimates above, we can make sure that
the constant $c$ is small with respect to $\|\alpha_1\|=1$ and
therefore the iteration of the last inequality yields
(\ref{bound-below}).

To prove (\ref{bound-der}), we note that
$\varPsi(\rho,r)=\widetilde{U}^*(\rho,r)$ solves
\[
\partial_\rho \varPsi(\rho,r)=-\left[
\frac{B_2(\rho)}{\rho^2} -i
\widetilde{V}(\rho)\right]\varPsi(\rho,r),\quad \varPsi(r,r)=I,
\quad \rho<r
\]
 Since $\widetilde{V}$ is compactly supported, we can take
 $r\to\infty$ and consider
 $w(\rho)=\varPsi (\rho,\infty)\eta$.
\begin{equation}
w'=-\left[ \frac{B_2(\rho)}{\rho^2} -i
\widetilde{V}(\rho)\right]w,\quad w(\infty)=\eta \label{core}
\end{equation}
Multiply the both sides by $w$ take the real part and integrate.
We have
\begin{equation}
\|w(t)\|^2+2\int\limits_t^\infty \left| \frac{\langle
B_2(s)w,w\rangle}{s^2} \right|ds=\|\eta\|^2, \quad 0<t<\infty
\label{petty}
\end{equation}
Then, multiplication of (\ref{core}) by $w'$ and integration from
$t$ to $\infty$ yields
\begin{equation}
\int\limits_{t}^{\infty}
\|w'(s)\|^2ds=-\int\limits_{t}^{\infty}\frac{\langle
B_2(s)w,w'\rangle}{s^2}ds+i\int\limits_{t}^{\infty} \langle
\widetilde{V}(s)w,w'\rangle ds \label{syuzhet}
\end{equation}
Denote the first term by $I$. Then, integration by parts gives
\[
\Re I=\frac 12\left[ \frac{\langle
B_2(t)w,w\rangle}{t^2}-2\int\limits_{t}^{\infty} \frac{\langle
B_2(s)w,w\rangle}{s^3}ds+\int\limits_{t}^{\infty} \frac{\langle
B_2'(s)w,w\rangle}{s^2}ds\right]
\]
The first term is nonpositive. For the second one, (\ref{petty})
yields
\[
\left|\int\limits_{t}^{\infty} \frac{\langle
B_2(s)w,w\rangle}{s^3}ds\right|<t^{-1}\|\eta\|^2
\]
The third term can be bounded by
\begin{equation}
\sum\limits_{l,m, r_m\geq t}
|w_l^m(r_m)|^2\frac{|\lambda_m|^2}{r_m^2}\leq
C\|\eta\|^2\sum\limits_{m>t^\alpha} m^{2-2/\alpha}=\bar{o}(1)
\|\eta\|^2 \label{chain}
\end{equation}
since $r_m\sim m^{1/\alpha}$ and we also used (\ref{petty}) once
again to estimate the sum in $l$ that corresponds to eigenspace of
each $\lambda_m$ for different values of $r_m$.

The second term in (\ref{syuzhet}) can be estimated by
Cauchy-Schwarz since $\|\widetilde{V}\|\in L^2[1,\infty)$. Taking
the real part of (\ref{syuzhet}) yields
\[
\int\limits_{t}^{\infty}
\|w'(s)\|^2ds<C\|\eta\|^2\left[\bar{o}(1)+\int\limits_t^\infty
\|\widetilde{V}(s)\|^2 ds\right]
\]
This inequality yields (\ref{bound-der}).
\end{proof}

We will need the following statement later on. Recall that
$\widetilde{V}_{(m)}(r)=\widetilde{V}(r)\cdot \chi_{r>m}$.
\begin{lemma}
Let $f(r)\in L^2(\mathbb{R}^+,L^2(\Sigma))$  and $k\in
\mathbb{C}^+$. Introduce $\psi=[\widetilde{P}^0(k)]^{-1}f$ and
$\psi_{(m),b}=[\widetilde{P}_{(m),b}(k)]^{-1}f$. Then, for any
$\tau>0$,
\[
\|\psi_{(m),b}(\tau)-\psi(\tau)\|_{L^2(\Sigma)}\to 0, \,
\|\psi_{(m),b}'(\tau)-\psi'(\tau)\|_{L^2(\Sigma)}\to 0
\]\label{pet1}
\end{lemma}
\begin{proof}
From the second resolvent identity, we have
\[
\psi-\psi_{(m),b}=[\widetilde{P}_{(m),b}(k)]^{-1}\left[k\widetilde{V}_{(d)}\psi
-\frac{B_{2,b}(r)-B_2(r)}{r^2}\psi\right]
\]
Since $\psi\in H^2(\mathbb{R}^3)$, we have
\[
\frac{B}{r^2}\,\psi\in L^2(\mathbb{R}^+,L^2(\Sigma))
\]
and therefore
\[
\|\psi-\psi_{(m),b}\|_{L^2(\mathbb{R}^+,L^2(\Sigma))}\to 0
\]
as $m,b\to\infty$. Compare two equations
\[
-\frac{d^2}{dr^2} \psi-\frac{B_2(r)}{r^2}\psi=k^2\psi+f
\]
and
\[
-\frac{d^2}{dr^2}
\psi_{(m),b}-\frac{B_{2,b}(r)}{r^2}\psi_{(m),b}+k\widetilde{V}_{(m)}\psi_{(m),b}=k^2\psi_{(m),b}+f
\]

Now, the Theorem for traces of $H^2(\mathbb{R}^3)$ functions
written in spherical coordinates yields the statement of the
Lemma.
\end{proof}

Consider spherically symmetric function $f$ having support on
$0<r<1$ such that $\|f\|_2=1$. Let
$\psi^0(r)=(\widetilde{H}^0-k^2)^{-1}f$,
$\mu^0(r)=\exp(-ikr)\psi^{0}(r)$, and $A^0(k)=\lim_{r\to\infty}
\mu^0(r)$. Since $f$ is spherically symmetric, $\mu^{0}(r)$ is
spherically symmetric as well and $A^0(k)$ is nonzero function
entire in $k$.

\begin{lemma}
Let $f$ be spherically symmetric with support in $[0,1]$,
$\alpha=2/3-$ and $\gamma>3/2-\alpha$. Then, for any $k\in
\mathbb{C}^+$ which is not zero of $A^0(k)$, there is $d>0$ such
that
\begin{equation}
\|J_{\widetilde{V}_{(d),R},b}(k,\xi)\|>\delta(k,d,V,T_1,f)>0
\label{a-below}
\end{equation}
uniformly in $R>R_0$, $b>R^\alpha$,  and $|\xi|<T_1$.
\end{lemma}
\begin{proof}
Without loss of generality, we again assume that $k=i/2$,
$\xi=-2$. Then, the equation for $\mu$ can be rewritten as (we
suppress the dependence of $\mu$ on $d$, $b$, and $R$)
\[
\mu'=\frac{B_{2,b}(r)}{r^2}u +i
\widetilde{V}_{(d),R}\mu+\mu''+e^{r/2}f
\]
The support of $f$ is within the interval $(0,1)$ and we therefore
have
\[
\mu(r)=\widetilde{U}_{(d),R,b}(\tau,r) \mu(\tau)
+\int\limits_\tau^r \widetilde{U}_{(d),R,b}(\rho,r)
\mu''(\rho)d\rho
\]
By making $d$ and $b$ large, we can make sure that
$\mu_{(d),R,b}(\tau)$ is close to $\mu^0(\tau)$ uniformly in
$R>R_0$ (by Lemma \ref{pet1}). On the other hand, $\mu^0(\tau)\sim
A^0(i/2)\neq 0$ as $\tau$ is large. Then, the absolute value of
the first term can be controlled from below by the Lemma
\ref{twist}. The second term can be made arbitrarily small if
large $\tau, d, b$ are fixed and $r\to\infty$. Indeed, its limit
as $r\to\infty$ is equal to
\[
\int\limits_\tau^\infty
\widetilde{U}_{(d),R,b}(\rho,\infty)\mu''(\rho)d\rho =I_1+I_2
\]
where
\[
I_1=-\widetilde{U}_{(d),R,b}(\tau,\infty)\mu'(\tau)
\]
and
\[
I_2=-\int\limits_{\tau}^{\infty} \partial_\rho
\widetilde{U}_{(d),R,b}(\rho,\infty)\mu'(\rho)d\rho
\]
 By fixing $\tau$, $d$, and $b$ large ($\tau<d$), we can make $I_1$
arbitrarily small because $(\mu^0)'(\tau)$ tends to zero at
infinity and $\|\mu'(\tau)- (\mu^0)'(\tau)\|_2\to 0$ as
$d,b\to\infty$ (by Lemma \ref{pet1}). Thus, we are left only with
$I_2$ to estimate. We have
\[
\|I_2\|=\max_{\|\eta\|_{L^2(\Sigma)=1}}\left|\langle
\int\limits_{\tau}^{\infty} \partial_\rho
\widetilde{U}_{(d),R,b}(\rho,\infty)\mu'(\rho)d\rho,
\eta\rangle\right|\leq
\]
\[
\leq  \sup_{\|\eta\|_{L^2(\Sigma)=1}}\int\limits_{\tau}^{\infty}
\left|\langle \mu'(\rho),\partial_\rho
\widetilde{U}_{(d),R,b}^*(\rho,\infty)\eta\rangle\right| d\rho
\]

By Cauchy-Schwarz and (\ref{rad-der}), we have
\[
\|I_2\|\leq
C\sup_{\|\eta\|_{L^2(\Sigma)=1}}\left[\int\limits_\tau^{\infty}
\|\partial_\rho
\widetilde{U}_{(d),R,b}^*(\rho,\infty)\eta\|^2d\rho\right]^{1/2}
\]
and the last integral can be made arbitrarily small by choosing
$\tau$ large (see (\ref{bound-der})).
\end{proof}
{\it Proof of Theorem \ref{theorem-last}.} The proof repeats the
arguments given before. We have uniform control over $\|J_{R,b}\|$
provided by (\ref{always1}) and (\ref{a-below}). These estimates,
(\ref{another-fact}), and Lemma \ref{weak-c} allow to use the
subharmonicity argument to get necessary bounds for the entropy.
$\Box$

\section{Appendix: Combes-Thomas inequality.}
Many properties of $P(k)$ are similar to those of general
Schr\"odinger operators. For completeness of discussion on decay
of Green's function, we prove the analog of the so-called
Combes-Thomas inequality (see, e.g. \cite{kg}). It gives a general
uniform bound on Green's function of $P(k)$. We do not have to use
it to prove a.c. of the spectrum but we think it is interesting in
itself. For the next Theorem, we assume $\xi=1$.
\begin{theorem}
Let $V(x)\in L^\infty(\mathbb{R}^3)$ and $k\in \mathbb{C}^+$.
Then,
\begin{equation}
\|\chi_{|x-x_2|<1} P^{-1}(k)\chi_{|x-x_1|<1}\|_{2,2} \leq
C(k,\gamma)\exp(-\gamma |x_1-x_2|) \label{ct}
\end{equation}
for any $x_{1(2)}\in \mathbb{R}^3$ and any $\gamma\in (0,\nu\Im
k)$ (with $\nu$ -- some universal constant).
\end{theorem}
\begin{proof}
We use the standard weight. Consider any $a\in \mathbb{R}^3$ and
operator
\begin{equation}
P_a(k)=-\Delta-2a\nabla-|a|^2+k V-k^2=``e^{-ax}P(k)e^{ax}"
\label{new-a}
\end{equation}
on $H^2(\mathbb{R}^3)$. It is easy to show that this operator is
closed  and $P_a^*(k)=P_{-a}(\bar{k})$. The last equality in
(\ref{new-a}) is justified for, e.g., $H^2(\mathbb{R}^3)$
functions with compact support.

 Moreover, if $\|f\|=1$, then
\[
(P_a(k)f,f)=-(k-k_1)(k-k_2)
\]
where
\[
k_{1(2)}=\frac{c_1\pm \sqrt{c_1^2+4c_2}}{2}
\]
with
\[
c_1=\int V|f|^2dx, \quad c_2=\int |\nabla f|^2dx-|a|^2-2\int
a\nabla f \bar f dx
\]
Write
\[
c_1^2+4c_2=\alpha+i\beta
\]
where
\[
\beta=-8\Im \left[ \int a\nabla f \bar fdx\right]
\]
and
\[
\alpha=\left[\int V|f|^2dx\right]^2+4\int |\nabla f|^2dx-4|a|^2
\]
since
\[
\Re \int a\nabla f \bar fdx=0
\]
We are interested in the imaginary part of the square root of
$\alpha+i\beta$. The inequality  $\alpha\geq -4|a|^2$ is always
true. If $\alpha\leq 0$, then
\[
\int |\nabla f|^2dx\leq |a|^2
\]
and therefore
\[
|\beta|\leq 8|a|\cdot \|\nabla f\|_2\leq 8|a|^2
\]
So,
\[
\left|\Im \sqrt{\alpha+i\beta}\right|\leq
\left|\sqrt{\alpha+i\beta}\right|\leq (80)^{1/4} |a|
\]
If, on the other hand,  $\alpha>0$, then there is $\kappa\in
\mathbb{R}$, so that $\alpha=|a|^\kappa$ and
\[
\|\nabla f\|_2^2\leq |a|^2+|a|^\kappa/4
\]
For imaginary part of square root
\[
\left|\Im
\sqrt{\alpha+i\beta}\right|=\frac{2^{-1/2}|\beta|}{\left(\alpha+\sqrt{\alpha^2+\beta^2}\right)^{1/2}}
\]
and this function increases in $|\beta|$. Moreover,

\[
|\beta|\leq 8|a|\left( |a|^2+|a|^\kappa/4\right)^{1/2}
\]

Thus,
\[
\left|\Im \sqrt{\alpha+i\beta}\right|\leq
C\frac{|a|(|a|^2+|a|^\kappa)^{1/2}}
{\left(|a|^{2\kappa}+|a|^4+|a|^{2+\kappa}\right)^{1/4}}<C|a|
\]
Consequently,
\[
|\Im k_{1(2)}|\leq C_u|a|
\]
where $C_u$ is a universal constant (we believe more accurate
analysis should yield $C_u=1$). That implies, of course, that
$\sigma(P_a(k))$ lies inside the strip $|\Im k|<C_u|a|$.

\begin{lemma}
For any function $f\in L^2(\mathbb{R}^3)$ with compact support and
any $k$ outside the strip $|\Im k|<C_u|a|$, we have
\begin{equation}
\exp(ax) P^{-1}_a(k)\exp(-ax)f=P^{-1}(k)f \label{obvious}
\end{equation}
\end{lemma}
\begin{proof}
 Consider $\cal{R}={\rm Ran}\,
P(k)[\cal{L}]$, where $\cal{L}$ denotes the linear manifold of
$H^2(\mathbb{R}^3)$ functions with compact support. For any $f\in
\cal{R}$, (\ref{obvious}) is true just because the last equality
of (\ref{new-a}) is true for functions in $\cal{L}$.

Take arbitrary open ball $\Omega$ and those functions from
$\cal{L}$ that are supported inside $\Omega$. Denote the linear
manifold of these functions by $\cal{L}_{\Omega}$. Operator $P(k)$
defined on $\cal{L}_{\Omega}$ can be closed to
$P_\Omega(k)=-\Delta_0+kV-k^2$ with
$\cal{D}[P_\Omega(k)]=H^2_0(\Omega)$, where $-\Delta_0$ is Laplace
with Dirichlet b.c. on $\partial \Omega$. Let
Ran$P(k)[\cal{L}_{\Omega}]=$Ran$P_\Omega(k)[\cal{L}_{\Omega}]=\cal{R}_\Omega$.
Then, $\overline{\cal{R}_\Omega}=L^2(\Omega)$. All that can be
justified in the standard way.

Now, consider any $f\in L^2(\mathbb{R}^3)$ with support, say,
within some ball $\Omega$. One can find $f_n\in \cal{R}_\Omega$
such that $f_n\to f$ in $L^2(\Omega)$. Since any function from
$\cal{R}_\Omega$ continued to $\Omega^c$ as zero is also from
$\cal{R}$, (\ref{obvious}) is true for each $f_n$. On the other
hand, $f_n$ is supported within $\Omega$ and therefore
$\exp(-ax)f_n\to \exp(-ax) f$ in $L^2(\mathbb{R}^3)$. So,
\[
P^{-1}_a(k)\exp(-ax) f_n\to P^{-1}_a(k) \exp(-ax)f, \,
P^{-1}(k)f_n\to P^{-1}(k)f
\]
where the convergence is in $L^2(\mathbb{R}^3)$. Thus, for
arbitrary $h\in L^2(\mathbb{R}^3)$ with compact support
\[
\langle \exp(ax) P^{-1}_a(k)\exp(-ax)f_n, h\rangle\to \langle
\exp(ax) P^{-1}_a(k) \exp(-ax) f, h\rangle
\]
and
\[
\langle \exp(ax) P^{-1}_a(k) \exp(-ax) f,h\rangle =\langle
P^{-1}(k) f,h\rangle
\]
Since $h$ was arbitrary and $\exp(ax) P^{-1}_a(k) \exp(-ax) f$ is
in $L^2_{\rm loc}$ apriori, we have the statement of the Lemma.
\end{proof}
To finish the proof of the Theorem, assume $x_1=0$ without loss of
generality. Then, take $a=-|a|x_2/|x_2|$ with $|a|<\nu \Im k$,
$\nu=C_u^{-1}$. Let $f$ be any $L^2$ function supported within the
unit ball around $0$. One can then use Lemma and bound on
$\|P_a^{-1}(k)\|$ to get (\ref{ct}).
\end{proof}

\end{document}